\documentclass[10pt,a4paper,leqno]{amsart}
\usepackage{amsmath}
\usepackage{amsthm}
\usepackage{amsfonts}
\usepackage{amssymb}
\usepackage{graphicx}
\usepackage{amsbsy}
\usepackage{mathrsfs}
\usepackage{esint}
\usepackage{bbm}
\usepackage{dsfont}
\usepackage{bm}
\usepackage{color}
\usepackage{mathabx}
\usepackage{mathtools}
\usepackage{eucal}

\usepackage{epstopdf}

\usepackage{amsmath, amssymb, graphics, setspace}

\newcommand{\mathsym}[1]{{}}
\newcommand{\unicode}[1]{{}}

\newcommand{\be}{\begin{equation}}
\newcommand{\ee}{\end{equation}}
\newcommand{\ba}{\begin{array}}
\newcommand{\ea}{\end{array}}
\newcommand{\bea}{\begin{eqnarray}}
\newcommand{\eea}{\end{eqnarray}}
\newcommand{\bee}{\begin{eqnarray*}}
\newcommand{\eee}{\end{eqnarray*}}

\newcommand{\VV}{\mathcal{V}}
\newcommand{\WW}{\mathcal{W}}
\newcommand{\MM}{\mathcal{M}}
\newcommand{\FF}{\mathcal{F}}
\newcommand{\RR}{\mathcal{R}}
\newcommand{\NN}{\mathcal{N}}

\renewcommand{\AA}{\mathcal{A}}
\newcommand{\BB}{\mathcal{B}}
\newcommand{\LL}{\mathcal{L}}

\newtheorem{theorem}{Theorem}

\newtheorem{lemma}{Lemma}[section]
\newtheorem{prop}{Proposition}[section]

\newtheorem*{remark*}{Remark}
\newtheorem*{remarks*}{Remarks}

\numberwithin{equation}{section}

\catcode`@=11
\def\section{\@startsection{section}{1}%
  \z@{1.5\linespacing\@plus\linespacing}{.5\linespacing}%
  {\normalfont\bfseries\large\centering}}
\catcode`@=12

\newcommand{\R}{\mathbb{R}}

\newcommand{\C}{\mathbb{C}}
\newcommand{\N}{\mathbb{N}}
\newcommand{\DD}{\Delta}
\newcommand{\pt}{\partial}
\renewcommand{\leq}{\leqslant}
\renewcommand{\geq}{\geqslant}

\newcommand{\eps}{\varepsilon}

\newcommand{\cO}{\mathcal O}

\renewcommand{\epsilon}{\varepsilon}

\renewcommand{\phi}{\varphi}

\begin{document}

\title{Blowup for Biharmonic NLS}

\author[T. Boulenger]{Thomas Boulenger}
\address{University of Basel, Department of Mathematics and Computer Science, Spiegelgasse 1, CH-4051 Basel, Switzerland.}
\email{thomas.boulenger@unibas.ch}

\author[E. Lenzmann]{Enno Lenzmann}
\address{University of Basel, Department of Mathematics and Computer Science, Spiegelgasse 1, CH-4051 Basel, Switzerland.}
\email{enno.lenzmann@unibas.ch}

\begin{abstract}
We consider the Cauchy problem for the biharmonic (i.\,e.~fourth-order) NLS with focusing nonlinearity given by
$$
i \partial_t u = \Delta^2 u - \mu \Delta u -|u|^{2 \sigma} u  \quad \mbox{for $(t,x) \in [0,T) \times \R^d$},
$$
where $0 < \sigma <\infty$ for $d \leq 4$ and $0 < \sigma \leq 4/(d-4)$ for $d \geq 5$; and $\mu \in \R$ is some parameter to include a possible lower-order dispersion. In the mass-supercritical case $\sigma > 4/d$, we prove a general result on finite-time blowup for radial data in $H^2(\R^d)$ in any dimension $d \geq 2$. Moreover, we derive a universal upper bound for the blowup rate for suitable $4/d < \sigma < 4/(d-4)$. In the mass-critical case $\sigma=4/d$, we prove a general blowup result in finite or infinite time for radial data in $H^2(\R^d)$. As a key ingredient, we utilize the time evolution of a nonnegative quantity, which we call the (localized) Riesz bivariance for biharmonic NLS. This construction provides us with a suitable substitute for the variance used for classical NLS problems. 

In addition, we prove a radial symmetry result for ground states for the biharmonic NLS, which may be of some value for the related elliptic problem.
\end{abstract}


\maketitle

\section{Introduction and Main Results}

In this paper, we consider the Cauchy problem for the biharmonic (i.\,e.~fourth-order) NLS with focusing power-type nonlinearity given by
\be \label{eq:bNLS}
\left\{\begin{aligned}
   & i \partial_t u = \DD^2 u - \mu \DD u - |u|^{2\sigma} u, \\
	 & u(0,x) = u_0(x) \in H^2(\R^d), \quad u : [0,T) \times \R^d \to \C,
\end{aligned}\right.
\ee
where  $0 < \sigma < \infty$ for $d  \leq 4$ and $0 < \sigma \leq \frac{4}{d-4}$ for $d \geq 5$. Here the parameter $\mu \in \R$ allows us to include a possible lower-order dispersion of classical NLS type. 

The biharmonic NLS provides a canonical model for nonlinear Hamiltonian PDEs with dispersion of super-quadratic order. Historically, the study of  biharmonic NLS goes back to Karpman and Karpman-Shagalov \cite{Ka-96,KaSh-00} in the physics literature, followed by the work of Fibich-Ilan-Papanicolaou \cite{FiIlPa-02}, where the rigorous analysis of these models was initiated. In recent years, a considerable amount of work has been devoted to the study of \eqref{eq:bNLS}. For instance, we refer to the works by Ben-Artzi-Koch-Saut \cite{BeKoSa-00} and Pausader \cite{Pa-07,Pa-09,Pa-09b} on well-posedness and scattering for biharmonic NLS; see also \cite{PaSh-10, MiXuZh-09, PaXi-13}.

Despite the fact that problem \eqref{eq:bNLS} bears a lot of resemblance to the classical NLS, several key questions have been out of scope by rigorous analysis up to now. Here, as a chief open problem addressed in this paper, we mention the existence of blowup solutions for problem \eqref{eq:bNLS}, which has been strongly supported by a series of numerical studies done by Fibich and coworkers \cite{BaFiMa-10, BaFiMa-10b, BaFi-11} for mass-critical and mass-supercritical powers $\sigma \geq 4/d$. In the present paper, we shall give an affirmative answer to the existence of blowup solutions for radial data in $H^2(\R^d)$ satisfying criteria that appear natural from known results on blowup for NLS and nonlinear wave equations (NLW). As another main result, we also derive a universal upper bound on the blowup rate in the mass-supercritical case for suitable exponents $\sigma > 4/d$.


Before we turn to the statement of the main results, let us  mention some general features of the evolution problem considered in this paper. Similar to the classical NLS, equation (\ref{eq:bNLS}) can be viewed as an infinite-dimensional Hamiltonian system, which enjoys the conservation of mass $M[u]$ and energy $E[u]$ that are given by
\be
    M[u] = \int_{\R^d} |u|^2\,dx , \\
\ee
\be
  E[u] = \frac12 \int_{\R^d} |\Delta u|^2\,dx + \frac{\mu}{2} \int_{\R^d} |\nabla u|^2\,dx - \frac{1}{2\sigma + 2} \int_{\R^d} |u|^{2\sigma + 2}\,dx.
\ee
Let us emphasize the fact that \eqref{eq:bNLS} does not possess any Galilean or Lorentz symmetry in contrast to classical NLS or NLW, respectively. With regard to classification of the criticality level for problem \eqref{eq:bNLS}, let us  define the number
\be
s_c := \frac{d}{2} - \frac{2}{\sigma}.
\ee
If we suppose for the moment that $\mu=0$ holds in \eqref{eq:bNLS}, we have the exact scaling invariance so that $u(t,x)$ can be mapped to another solution given by
\be
u_\lambda(t,x) = \lambda^{\frac{d}{2}-s_c} u(\lambda^4 t, \lambda x) \quad \mbox{with $\lambda > 0$}.
\ee
This rescaling preserves the homogeneous $\dot{H}^{s_c}$-norm of the original solution $u(t)$. Note that $s_c=2$ corresponds to the endpoint case $\sigma = \frac{4}{4-d}$ in \eqref{eq:bNLS} for dimensions $d \geq 5$. In view of the conservation laws above, we refer to the cases $s_c < 0$, $s_c=0$, and $s_c>0$ as {\em mass-subcritical, mass-critical,} and {\em mass-supercritical,} respectively. The endpoint case $s_c=2$ is {\em energy-critical}. Note that the cases $s_c=0$ and $s_c=2$ correspond to the exponents $\sigma=4/d$ and $\sigma=4/(d-4)$ in problem \eqref{eq:bNLS}, respectively.

From \cite{Pa-07} we recall the local well-posedness of the Cauchy problem \eqref{eq:bNLS} holds for $s_c \leq 2$. Furthermore, if $s_c < 2$, we have the following blowup alternative: Either the solution $u\in C^0([0,T); H^2(\R^d))$ of \eqref{eq:bNLS} extends to all times $t \geq 0$, or we have that
$$
\lim_{t \uparrow T} \| \DD u(t) \|_{L^2} = +\infty
$$
 for some finite time $0 < T < +\infty$. In the energy-critical case $s_c=2$, we have a blowup alternative that involves a critical Strichartz norm in space-time; see Theorem \ref{thm:energy} below for more details. 
 
Finally, we mention that, in the mass-subcritical case $s_c < 0$, the conservation laws for $M[u]$ and $E[u]$ together with an interpolation estimate (see \eqref{ineq:GNS} below) imply that all solutions $u(t)$ of problem \eqref{eq:bNLS} extend to all times, and thus blowup cannot occur in the mass-subcritical case $s_c < 0$ in analogy to well-posedness theory for classical NLS.  The present paper will show that, for $s_c \geq 0$, we do have blowup for biharmonic NLS for radial solutions in $H^2$ that satisfy suitable criteria.

\subsection{Blowup for Mass-Supercritical Case}
First, we discuss the case of mass-supercritical powers in \eqref{eq:bNLS} below the energy-critical level, i.\,e., we suppose that
$$
0 < s_c < 2.
$$
In view of the conservation laws for mass and energy, we recall the Gagliardo--Nirenberg (GN) interpolation inequality
\be \label{ineq:GNS}
\| u \|_{L^{2 \sigma+2}}^{2 \sigma+2} \leq C_{d, \sigma} \| \DD u \|_{L^2}^{\frac{\sigma d}{2}}  \| u \|_{L^2}^{2- \frac{\sigma}{2}(d-4)}
\ee
valid for all $u \in H^2(\R^d)$ and where $C_{d, \sigma} > 0$ denotes the optimal constant; we refer to Appendix \ref{sec:Q} for more details. It is known that \eqref{ineq:GNS} has  optimizers $Q \in H^{2}(\R^d)$, which we refer to as {\em ground states} throughout the following. By rescaling, we can assume that any such ground state $Q \in H^2(\R^d)$ solves the nonlinear elliptic equation
\be \label{eq:Q}
\DD^2 Q  + Q -|Q|^{2 \sigma} Q = 0 \quad \mbox{in $\R^d$}.
\ee
We remark that uniqueness of $Q$ (modulo translation and phase) is not known. In fact, to the best of our knowledge, it is has not even been known whether $Q$ can be chosen radially symmetric, since classical methods (e.\,g.,~moving planes or rearrangement techniques in $x \in \R^d$) are not applicable for equation \eqref{eq:Q} due to the presence of the biharmonic operator $\DD^2$. But if we assume that $\sigma \in \N$ holds, we show that $Q$ can always be chosen to be radially symmetric and real-valued, by using rearrangement techniques in Fourier space; see Appendix \ref{sec:Q} for more details. Actually, we will not make use of this fact shown here. But this symmetry result for ground states $Q$ seems to be new and it is perhaps of some independent value.
 
Our first main result gives sufficient criteria for finite-time blowup for \eqref{eq:bNLS} in the class of radial initial data.

\begin{theorem}[Blowup for Mass-Supercritical Case] \label{thm:blowup} Let $d \geq 2$, $\mu \in \R$, and $0 < s_c < 2$ with $\sigma \leq 4$. Suppose that $u_0 \in H^2(\R^d)$ is radial and satisfies one of the following conditions.
\begin{enumerate}
\item[(i)] If $\mu \neq 0$, we assume that 
$$
E[u_0] <  \begin{dcases*} 0 & for $\mu > 0$, \\  -\varkappa \mu^2 M[u_0]  & for $\mu < 0$, \end{dcases*}
$$
with some constant $\varkappa=\varkappa(d,\sigma)  > 0$.
\item[(ii)] If $\mu=0$, we assume that either $E[u_0] < 0$ or, if $E[u_0] \geq 0$, we suppose that
$$
E[u_0]^{s_c} M[u_0]^{2-s_c} < E[Q]^{s_c} M[Q]^{2-s_c},
$$
and
$$  
\| \DD u_0 \|_{L^2}^{s_c} \| u_0 \|_{L^2}^{2-s_c} > \| \DD Q \|_{L^2}^{s_c} \| Q \|_{L^2}^{2-s_c} .
$$
\end{enumerate}  
Then the solution $u \in C([0,T); H^2(\R^d))$ of \eqref{eq:bNLS} blows up in finite time, i.\,e., we have $0 < T < +\infty$ and $\lim_{t \uparrow T} \| \DD u(t) \|_{L^2} = +\infty$.
\end{theorem}

\begin{remarks*}
{\em
1. The extra condition $\sigma \leq 4$ arises from the use of the Strauss inequality (i.\,e., a radial Sobolev inequality) in $\R^d$ with $d \geq 2$.  An analogous condition on the exponent $\sigma$ appears in the blowup proof of Ogawa and Tsutsumi \cite{OgTs-91} for classical NLS.

2. Note that if $\mu \geq 0$, the negative energy condition $E[u_0] < 0$ is sufficient. 

3. By time reversal symmetry, the equivalent blowup result holds for negative times.

4. For $0 < s_c < 2$ and initial data $u_0 \in H^2(\R^d)$ (which are not necessarily radial) with energy $E(u_0) \geq 0$ such that
$$ 
E[u_0]^{s_c} M[u_0]^{2-s_c} < E[Q]^{s_c} M[Q]^{2-s_c},
$$
and
$$
  \| \DD u_0 \|_{L^2}^{s_c} \| u_0 \|_{L^2}^{2-s_c} < \| \DD Q \|_{L^2}^{s_c} \| Q \|_{L^2}^{2-s_c},
  $$ 
  the corresponding solution $u \in C^0([0,\infty); H^2(\R^d))$ of \eqref{eq:bNLS} exits for all times $t \geq 0$ with an a-priori bound $\sup_{t \geq 0} \| \DD u(t) \|_{L^2} < +\infty$. This is a consequence of the conservation laws for mass and energy combined with the sharp version of the GN-inequality \eqref{ineq:GNS}. Note that quantities $E[Q]^{s_c} M[Q]^{2-s_c}$ and $\| \DD Q \|_{L^2}^{s_c} \| Q \|_{L^2}^{2-s_c}$ do not depend on the particular choice of a ground state $Q \in H^2(\R^d)$ solving \eqref{eq:Q} thanks to Pohozaev identities; see Appendix \ref{sec:Q}.

5. For $\sigma \in \N$, we show that ground states $Q=Q(|x|)$ can be chosen radial; see Appendix \ref{sec:Q}. In this case and with $\mu=0$ in \eqref{eq:bNLS}, we conclude that {\em solitary waves} $u(t,x) = e^{it} Q(x)$ are unstable due to nearby finite-time blowup solutions. Indeed, it is straightforward to check that radial initial data $u_0(|x|)= \lambda Q(|x|)$ with $\lambda > 1$  satisfy the assumptions of Theorem \ref{thm:blowup}. On the other hand, we deduce global-in-time existence for $u_0(|x|) = \lambda Q(|x|)$ when $\lambda < 1$ by the remark made above. Thus, in this case, the blowup conditions for radial $u_0 \in H^2(\R^d)$ are sharp.

6.~Similar blowup conditions for classical NLS involving products of suitable powers of $E[Q]$ and $M[Q]$ were derived in \cite{DuHoRo-08, HoRo-08}.

7.~In \cite{BaFi-11}, the authors investigate (by means of asymptotic analysis) self-similar blowup solutions for mass-supercritical biharmonic NLS. Assuming a conjecture to hold for the solvability of a certain nonlinear ODE for a self-similar blowup profile $S_B$, the results in \cite{BaFi-11} yield the existence of singular solutions $u_s(t,x)$ for \eqref{eq:bNLS} when $\mu=0$ and $\sigma > 4/d$; these proposed explicit singular solutions $u_s(t) \not \in L^2(\R^d)$ become singular in finite time in the space $L^{2 \sigma+2}(\R^d)$. It is an interesting open question to rigorously prove the existence of $S_B$ and to understand how a suitably perturbed profile of $S_B$ may lead to explicit finite-time blowup solutions in energy space.
}
\end{remarks*}

The next main result establishes a universal bound on the blowup rate in the class of radial data. The precise statement is as follows. 

\begin{theorem}[Universal Upper Bound on Blowup Rate] \label{thm:rate}
Suppose $d \geq 3$, $\mu \in \R$, and $0 < s_c < 2$ with $\sigma < \min \left \{ \frac{3}{d}+\frac 1 2,  \frac{6}{d} \right \}$. Let $u_0 \in H^2(\R^d)$ be radial and assume that the corresponding solution $u \in C([0,T); H^2(\R^d))$ of \eqref{eq:bNLS} blows up in finite time $0 < T < +\infty$. Then, for any time $t \in [0,T)$, we have the bound
$$
\int_{t}^T (T-\tau) \| \Delta u(\tau) \|_{L^2}^2 \, d \tau \leq C \left (T-t \right )^{\frac{2 \beta}{1+\beta}} 
$$
with some constants $C=C(u_0,d , \sigma)>0$ and $\beta=\beta(d, \sigma) > \alpha$, where
$$
 \alpha = \frac{4 - \sigma}{\sigma (d-1)}.
$$
Moreover, it holds that $\beta = \alpha + \cO(s_c)  \to \alpha$ as $s_c \to 0$.
\end{theorem} 

\begin{remarks*}
{\em 1.~Our strategy to prove Theorem \ref{thm:rate} is inspired by the remarkable proof of Merle-Rapha\"el-Szeftel \cite{MeRaSz-14}, where a (sharp) universal upper bound for the blowup rate for mass-supercritical classical NLS is established. However, the proof in \cite{MeRaSz-14} makes use of the variance algebra for classical NLS, which is not at our disposal for biharmonic NLS and hence cannot be directly adapted to the present situation. To overcome this, we introduce a suitable nonnegative quantity $\VV_{\psi_R}[u]$, which we refer to as the (localized) {\em Riesz bivariance}; see below for more information on this.

2.~We need to impose the extra condition $\sigma < \min \{ \frac{3}{d} + \frac{1}{2}, \frac{6}{d} \}$ in order to control certain nonlinear interaction terms (which are not present at all for classical NLS). See below for more details on this. Note that this technical assumption on $\sigma$ is automatically satisfied when $d \geq 12$, since we have $\sigma < \frac{4}{d-4}$ in the energy-subcritical case. 

3.~In the proof of Theorem \ref{thm:rate} given below, we give an explicit formula for $\beta=\beta(d, \sigma)> \alpha$; see Section \ref{sec:blowup_rate} for more details.

4.~The numerical analysis in \cite{BaFiMa-10} suggests that the sharp upper bound is $\beta = \alpha$. It seems a challenging open problem to prove this observation by rigorous means.
}
\end{remarks*}

\subsection{Blowup for Mass-Critical Case}

We now consider the mass-critical case $s_c = 0$ in \eqref{eq:bNLS}, i.\,e., we assume that $\sigma = 4/d$ holds. We have the following result on finite- and infinite-time blowup for radial data.

\begin{theorem}[Blowup for Mass-Critical Case] \label{thm:mass}
Let $d \geq 2$, $\mu \geq 0$, and $s_c=0$.  Let $u_0 \in H^2(\R^d)$ be radial with $E(u_0) < 0$. Then the solution $u \in C^0([0;T); H^2(\R^d))$ of problem \eqref{eq:bNLS} satisfies the following.
\begin{enumerate}
\item[(i)] If $\mu > 0$, then $u(t)$ blows up in finite time.
\item[(ii)] If $\mu =0$, then $u(t)$ either blows up in finite time or $u(t)$ blows up in infinite time such that
$$
\| \DD u(t) \|_{L^2} \geq C t^2 \quad \mbox{for $t \geq t_0$},
$$
with some constants $C=C(u_0) > 0$ and $t_0=t_0(u_0) > 0$. Moreover, in the latter case and for dimensions $d \geq 5$, it holds that 
$$
\limsup_{t \to +\infty} \left ( t^{-\nu} \| \DD u(t) \|_{L^2} \right ) = +\infty,
$$
for any  $\nu < \nu_*$, where
$$
\nu_* = \begin{dcases*} +\infty & for $d \geq 8$, \\ 24 & for $d=7$, \end{dcases*} \qquad \nu_* = \begin{dcases*} 10 & for $d=6$, \\ 4 & for $d=5$. 
\end{dcases*}
$$
\end{enumerate}
\end{theorem}

\begin{remarks*}
{\em
1. When $\mu > 0$, the proof is a slight modification of the proof of Theorem \ref{thm:blowup} and exploits the fact that the exponent $\sigma = \frac{4}{d}$ is ``mass-supercritical'' with respect to the lower-order NLS type dispersion $-\mu\DD$. On the other hand, we are presently not able to deal with the case $\mu < 0$.

2. For $\mu=0$ and $\sigma=4/d$, equation \eqref{eq:bNLS} becomes invariant under $L^2$-mass preserving rescaling. In this case, the analysis turns out to be much more delicate, and we are currently not able to conclude that radial negative energy solutions must blowup in finite time. The proof for the growth estimate utilizes the localized Riesz bivariance.

3. This blowup result for $\mu=0$ complements the analysis of Pausader and Shao \cite{PaSh-10}, where global-in-time well-posedness for radial initial data $u_0 \in L^2(\R^d)$ with $\| u_0 \|_{L^2} < \| Q \|_{L^2}$, which implies that $E[u_0] > 0$, was shown by implementing the Kenig--Merle methodology \cite{KeMe-06}. 

4. In view of well-known blowup results for negative energy data for focusing mass-critical NLS, it seems natural to conjecture that we always have finite-time blowup for $\mu=0$. 

5.~ Lower bounds on blowup rates (as dictated by local well-posedness), convergence properties to a blowup profile (given by $Q$), and $L^2$-mass concentration were shown in  \cite{BaFiMa-10} for finite-time blowup $H^2$-solutions for the $L^2$-critical biharmonic NLS \eqref{eq:bNLS} with $\mu=0$ and $\sigma=4/d$. These results are in direct analogy to known results for $L^2$-critical classical NLS. In particular, the proofs in \cite{BaFiMa-10} follow from an adaptation of arguments in \cite{MeTs-90, We-89} developed for $L^2$-critical NLS. 

6.~After finalizing this paper, we learned from the recent work by Cho et al.~\cite{ChOzWa-14}, where existence of finite-time blowup solutions for fourth-order $L^2$-critical NLS of the specific form $i \pt_t u = (\alpha \DD^2 - \DD ) u - |x|^{-2} |u|^{\frac{4}{d}} u$ with $\alpha > 0$ was shown for sufficiently high space dimensions $d$, by means of a (non-localized) virial/variance type argument. For local nonlinearities, the arguments used in \cite{ChOzWa-14} strongly exploit the fact that the nonlinearity is of the form $-\rho(x) |x|^{-2} |u|^{\frac{4}{d}} u$ with non-increasing radial $\rho(x)$.
}
\end{remarks*}

\subsection{Blowup for Energy-Critical Case}

As the final main result in this paper, we turn to the energy-critical case $s_c =2$, i.\,e., we assume that $d \geq 5$ holds and choose $\sigma = \frac{4}{4-d}$. For this endpoint case, we recall the homogeneous Sobolev inequality
\be \label{ineq:sob}
\| u \|_{L^{\frac{2d}{d-4}}} \leq C_d \| \DD u \|_{L^2} ,
\ee
valid for all $u \in \dot{H}^2(\R^d)$ and where $C_d > 0$ denotes the optimal constant. It is a classical result that inequality \eqref{ineq:sob} has an optimizer $W \in \dot{H}^2(\R^d)$ that is unique (up to scaling and translation). In particular, it is known that $W \in \dot{H}^2(\R^d)$ is radial, nonnegative and it solves the nonlinear elliptic equation
\be \label{eq:W}
\DD^2 W - |W|^{\frac{8}{d-4}} W = 0 \quad \mbox{in $\R^d$}.
\ee
In fact, we have the explicit formula 
\be
W(x) =\left (  \frac{(d(d-4)(d^2-4))^{\frac{1}{4}}}{1 +x^2} \right )^{\frac{d-4}{2}} .
\ee 
As an aside, we remark that $W \not \in L^2(\R^d)$ for $5 \leq d \leq 8$ due to its slow algebraic decay at infinity. The reason why ground states for \eqref{eq:W} are much better understood than for the elliptic problem \eqref{eq:Q} is due to the conformal invariance of equation \eqref{eq:W}. 
 
We have the following blowup result for the energy-critical case, which is a close variant of Theorem \ref{thm:blowup} above.

\begin{theorem}[Blowup for Energy-Critical Case] \label{thm:energy}
Let $d \geq 5$, $\mu \in \R$, and $s_c = 2$. Suppose that $u_0 \in H^2(\R^d)$ is radial and satisfies one of the following properties.
\begin{enumerate}
\item[(i)] If $\mu \neq 0$, we assume that 
$$
E[u_0] <  \begin{dcases*} 0 & for $\mu > 0$, \\   -\varkappa \mu^2 M[u_0] & for $\mu < 0$, \end{dcases*}
$$
with some constant $\varkappa=\varkappa(d) > 0$.
\item[(ii)] If $\mu=0$, we assume that either $E[u_0] < 0$ or, if $E[u_0] \geq 0$, we suppose that
$$
E[u_0] < E[W] \quad \mbox{and} \quad \| \DD u_0 \|_{L^2} > \| \DD W \|_{L^2}.
$$
\end{enumerate}
Then the solution $u \in C^0([0,T); H^2(\R^d))$ blows up in finite time, i.\,e., it holds that $0 < T < +\infty$ and
$$
\int_{0}^T \int_{\R^d} |u(t,x)|^{\frac{2(d+4)}{d-4}} \,d x \, dt = +\infty .
$$
\end{theorem}

\begin{remark*}
{\em This blowup result complements the works on the focusing energy-critical biharmonic NLS in \cite{Pa-09, MiXuZh-09}, where global-in-time well-posedness in $H^2(\R^d)$ for radial data with $E[u_0] < E[W]$ and $\| \DD u_0 \|_{L^2} < \| \DD W \|_{L^2}$ is established by implementing the Kenig--Merle rigidity method (see, e.\,g., \cite{KeMe-06}) for biharmonic NLS.
}
\end{remark*}

\subsection{Comments on the Proofs}
Let us give some explanations about the strategies behind the proofs in this paper, which are based on exploiting (localized) virial and variance-type identities for the biharmonic NLS. To simplify the following discussion, we suppose that the lower-order dispersion term is absent in \eqref{eq:bNLS}, i.\,e., we assume that 
$$
\mu=0.
$$

We begin with some formal observations. To this end, we suppose that $u=u(t,x)$ is a sufficiently regular and spatially localized solution of \eqref{eq:bNLS} for the following quantities to make sense. Then, as a simple consequence of the exact scaling behavior, we formally obtain the {\em virial law} given by
\be \label{eq:virial_intro}
\frac{d}{dt} \left ( 2 \, \mathrm{Im} \int_{\R^d} \overline{u}(t) x \cdot \nabla u(t) \, dx \right )= 4 d \sigma E[u_0] - (2 d \sigma -8) \| \DD u(t) \|_{L^2}^2.
\ee
In addition, a calculation shows that the nonnegative quantity
\be
\VV[u(t)] :=  \| |\nabla|^{-1} x u(t) \|_{L^2}^2 = \int_{\R^d} \overline{u}(t) x \cdot (-\DD)^{-1} x u(t) \, dx 
\ee
formally satisfies the differential law
\be \label{eq:variance_intro}
\frac{d}{dt} \VV[u(t)] = 8 \, \mathrm{Im} \int_{\R^d} \overline{u}(t) x \cdot \nabla u(t) \, dx + \mathrm{Error}[u(t)],
\ee
where $\mathrm{Error}[(t)]$ denotes some error term due to the nonlinearity in equation \eqref{eq:bNLS}. When combined with the virial law \eqref{eq:virial_intro}, this identity turns out to be a viable substitute for the {\em variance law} used for classical NLS. Since the quantity $\VV[u]$ scales like the fourth moment $\int |x|^4 |u(t)|^2$, we refer to $\VV[u]$ as the {\em Riesz bivariance} for the biharmonic NLS. As an aside, we remark that the use of the fourth moment $\int |x|^4 |u(t)|^2$ itself (or localized versions thereof) do not seem to give any insight, which was already pointed out in \cite{BaFiMa-10}. To conclude our formal discussion, we remark that for a (sufficiently regular and localized) solution $v(t,x)$ of the {\em free} biharmonic Schr\"odinger equation $i \pt_t v = \DD^2 v$, we can combine the identities in \eqref{eq:virial_intro} and \eqref{eq:variance_intro} to obtain the conservation law
$$
\left \| \left ( |\nabla|^{-1} x +4  i t \nabla |\nabla| \right ) v(t) \right \|_{L^2}^2 = \mbox{const.}
$$
which is an analogue to the celebrated {\em pseudo-conformal law}  for classical NLS (see \cite{GiVe-79}).

Let us now explain how to rigorously exploit the formal identities above for the nonlinear biharmonic NLS in some detail.  
The proofs of Theorems \ref{thm:blowup} and \ref{thm:energy}, which address the mass-supercritical case $s_c >0$, are inspired by a strategy that was introduced by Ogawa and Tstutsumi \cite{OgTs-91} to show blowup for radial solutions for mass-supercritical NLS with radial data $u_0 \in H^1(\R^d)$ with infinite variance (i.\,e., we may have $x u_0 \not \in L^2(\R^d)$). The adaptation of this argument to biharmonic NLS requires a careful analysis of the time evolution for the localized virial quantity
\be
\MM_{\phi_R}[u(t)] = 2 \, \mathrm{Im} \int_{\R^d} \overline{u}(t) \nabla \phi_R \cdot \nabla u(t) \, dx,
\ee  
Here $\phi_R(r)$ is a suitably chosen radial cutoff functions with $\nabla \phi_R(x) \equiv x$ for $|x| \leq R$ and $\nabla \phi_R(x) \equiv \mbox{const}.$ for $|x| \gg R$. Imposing the assumptions of Theorem \ref{thm:blowup} and recalling that we assume $\mu=0$ for simplicity, we obtain the differential inequality
\be \label{ineq:MR_intro}
\frac{d}{dt} \MM_{\phi_R}[u(t)] \leq 4 d \sigma E[u_0] - \left ( 2 \delta + o_R(1) \right ) \| \DD u(t) \|_{L^2}^2 + o_R(1)  
\ee
with $\delta = d \sigma - 4>0$ and error terms $o_R(1) \to 0$ as $R \to \infty$ uniformly in $t$. In fact, such an upper bound for time evolution for $\MM_{\phi_R}[u(t)]$ is reminiscent to blowup proofs for classical NLS (see \cite{OgTs-91}) and finite time blowup follows by integrating \eqref{ineq:MR_intro} and ODE comparison. But due to the presence of the biharmonic operator $\DD^2$ here, the calculational efforts to arrive at such an inequality requires some work that makes use of commutator identities. Let us also mention that \cite{PaSh-10,MiXuZh-09,Pa-07, Pa-09} have already made use of a localized virial quantity for biharmonic NLS with less detail. However, the point here is work out the signs of certain errors terms, which turn out to be essential when proving a blowup result.

On the other hand, the proof of Theorem \ref{thm:rate} and parts of Theorem \ref{thm:mass} both depend on a new ingredient, which is perhaps the most interesting aspect of this work. Here we introduce the localized version of the Riesz bivariance defined as 
\be
\VV_{\psi_R}[u(t)] := \int_{\R^d} \overline{u}(t) \nabla \psi_R \cdot  (-\DD)^{-1} \nabla \psi_R u(t) \,d x = \left \| |\nabla|^{-1} \left  \{ \nabla \psi_R u(t) \right \} \right \|_{L^2}^2,
\ee
with some cutoff function such that $\nabla \psi_R(x) \equiv x$ for $|x| \leq R$ and $\psi_R(x) \equiv \mbox{const}.$ for $|x| \gg R$. 
A subtle fact to be kept in mind is that  the cutoff function $\psi_R(r)$ appearing in the definition of $\VV_{\psi_R}[u]$ is {\em not} identical to $\phi_R$ used in the localized virial $\MM_{\phi_R}[u]$. Instead, these cutoff functions are related via the nonlinear equation $\partial_r \psi_R(r) = \sqrt{2 \phi_R(r)}$. A calculation then yields an identity of the form
\be \label{eq:var_intro}
\frac{d}{dt} \VV_{\psi_R}[u(t)] = 4 \MM_{\phi_R}[u(t)] + \NN_R[u(t)] + \cO(1),
\ee
where the commutator term
$$
\NN_R[u(t)] =  \int_{\R^d} \overline{u}(t) \left [-i|u(t)|^{2 \sigma},  \nabla \psi_R \cdot (-\DD)^{-1} \nabla \psi_R \right ] u(t)  \, dx,
$$
with $[X,Y] \equiv XY- YX$ arises from the nonlinearity in \eqref{eq:bNLS}. Compared to classical NLS, the presence of $\NN_R[u]$ substantially complicates the analysis. However, by exploiting the radial symmetry of $u$, we are able to derive certain bounds on $\NN_R[u]$ that will be essential in the proofs of Theorems \ref{thm:rate} and \ref{thm:mass} below. With the nonnegative quantity $\VV_{\psi_R}[u]$ and suitable bounds on $\NN_R[u]$ at our disposal, we are in the position to implement the remarkable strategy of Merle-Rapha\"el-Szeftel \cite{MeRaSz-14} (developed for mass-supercritical NLS) to obtain the universal upper bounds on blowup rates for biharmonic NLS in Theorem \ref{thm:rate} above. 

As a further application of the Riesz bivariance $\VV_{\psi_R}[u]$, we obtain the quantitative lower bounds on the infinite-time blowup rates in Theorem \ref{thm:mass} for the delicate case $\mu=0$. In particular, for dimensions $d \geq 8$ and the mass-critical exponent $\sigma=4/d$, the term $\NN_R[u]$ is ``almost'' controlled by $L^2$-mass conservation, since we find that
$$
\left | \NN_R[u(t)] \right | \lesssim_\eps   R^\eps \| \DD u(t) \|^{\frac{\eps}{2}} \| u_0 \|_{L^2}^{\frac{8}{d} + 1}
$$
With the help of this bound, we deduce that radial infinite-time blowup solutions $u(t)$ with $E[u_0] < 0$ in the mass-critical case and dimensions $d\geq 8$ must grow (at least along subsequences $t_n \to +\infty$) faster than any polynomial in $t$.

\subsection{Outlook and Future Problems}
We think that this paper contains many points of departure for future work. Let us briefly mention some of them as follows.

Of course, it would be desirable to remove  the radial symmetry assumption in $\R^d$. In fact, both the localized virial and Riesz bivariance identities hold true without imposing radiality. However, at the moment, it is not clear to us how to effectively control the error terms without radial symmetry. However, if we consider the biharmonic NLS \eqref{eq:bNLS} posed on a bounded domain $\Omega \subset \R^d$, we are able to remove the radiality assumption for the existence of blowup solutions, as shown in our companion paper \cite{BoLe-15}. But the case of non-radial data in $\R^d$ seems to be a challenging open problem. 

Furthermore, it seems natural to conjecture that finite-time blowup always occurs in the setting of Theorem \ref{thm:mass}, at least in sufficiently high dimensions. Another open problem that seems worthwhile attacking is to try to improve that upper bounds in Theorem \ref{thm:rate} to the rate $\beta=\alpha$, which is strongly indicated by numerics (see \cite{BaFiMa-10, BaFiMa-10b}). So far, the fact that we can only conclude that $\beta \geq \alpha$ is due to the bounds derived for $\NN_R[u]$. We may speculate that, by exploiting delicate cancellations and sign properties in the commutator term $\NN_R[u]$, that one may eventually prove that $\beta=\alpha$ holds. 

Another line of future research would be to study the blowup dynamics of collapsing solutions close to ground state solitary waves $u(t,x) = e^{it} Q(x)$ for the biharmonic NLS \eqref{eq:bNLS} for $\mu=0$. Here, as a starting point, a much better understanding of the related nonlinear elliptic problem \eqref{eq:Q} is needed (e.\,g., a proof of non-degeneracy and uniqueness of ground states).  

Finally, we think that the strategies developed in this paper can be extended (with some effort) to polyharmonic and fractional NLS of the form
\be \label{eq:poly}
i \partial_t u = (-\DD)^{s} u - |u|^{2 \sigma} u \quad \mbox{with $(t,x) \in \R \times \R^d$},
\ee
where $s \in \N$ is an integer (polyharmonic case) or $s >0$ is a non-integer number (fractional case); see \cite{BoHiLe-15}. A formal computation shows that the corresponding (localized) variance-type quantity for equation \eqref{eq:poly} is found to be
\be
\VV_{\psi_R}^{(s)}[u(t)] = \int_{\R^d} \overline{u}(t) \nabla \psi_R \cdot (-\DD)^{-s+1} \nabla \psi_R u(t)  \, dx
\ee
with $\nabla \psi_R(x) = x$ for $|x| \leq R$ and $\nabla \psi_R(x) = \mbox{const}.$ for $|x| \gg R$. Note that in the half-wave case $s=1/2$ and with mass-critical Hartree type nonlinearity, the nonlocalized version of $\VV_{\psi_R}^{(s)}$ (i.\,e.,~we replace $\nabla \psi_R$ by the unbounded function $x$) was used by Fr\"ohlich and Lenzmann \cite{FrLe-07} to prove finite-time blowup for radial solutions of the Boson star equation. 

\subsection*{Acknowledgments}
The authors gladly acknowledge financial support from the Swiss National Science Foundation (SNF) through Grant No.~200021--149233. We thank Y.~Cho for bringing us the reference \cite{ChOzWa-14} to our attention. Furthermore, we are grateful to G.~Baruch, G.~Fibich, and E.~Mandelbaum for pointing out to their results obtained in \cite{BaFi-11, BaFiMa-10, BaFiMa-10b}.

\section{Preliminaries and Plan of the Paper}

For later use, we recall the following radial Sobolev inequality found by Strauss \cite{St-77}: For every radial function $u \in H^1(\R^d)$ with $d \geq 2$, we have the pointwise bound
\be \label{ineq:Strauss}
|x|^{\frac{d-1}{2}} |u(x)| \leq 2 \| u \|_{L^2}^{\frac 1 2} \| \nabla u \|_{L^2}^{\frac 1 2} \leq 2 \| u \|_{L^2}^{\frac 3 4} \| \DD u \|_{L^2}^{\frac 1 4} \quad \mbox{for $x \neq 0$},
\ee
where for the second inequality we additionally assume that $u \in H^2(\R^d)$ holds; we refer to \cite{Ca-03} for a simple proof of the first inequality; the second inequality is a direct consequence of the fact that $\| \nabla u \|_{L^2} \leq \| u \|_{L^2}^{\frac 1 2} \| \DD u \|_{L^2}^{\frac 1 2}$ for $u \in H^2(\R^d)$. 

Throughout this paper, we make the standard abuse of notation by writing $f=f(r)$ with $r=|x|$ for a radial function $f: \R^d \to \C$. Moreover, we use  the convention that we sum over repeated indices from $1$ to $d$, e.\,g., we have $x_k y_k = \sum_{k=1}^d x_k y_k$ etc. Furthermore, we shall write
$$
X \lesssim Y
$$
to denote that $X \leq CY$ holds with some constant $C > 0$ that depends only on $d$, $\sigma$, and the radial cutoff function $\phi:�\R^d \to \R$  introduced in Section \ref{sec:locvirial} below.

This paper is organized as follows. In Section \ref{sec:locvirial}, we derive a localized virial identity for the biharmonic NLS. In Section \ref{sec:blowup_super}, we will prove Theorems \ref{thm:blowup} and \ref{thm:energy}. The localized Riesz bivariance identity for the biharmonic NLS is derived in Section \ref{sec:locrieszvar}. In Sections \ref{sec:blowup_rate} and \ref{sec:blowup_mass}, we give the proofs of Theorems \ref{thm:rate} and \ref{thm:mass}, respectively. 

\section{Localized Virial Identity} \label{sec:locvirial}

Let $\phi : \R^d \to \R$ be a radial function with regularity property $\nabla^j \phi \in L^\infty(\R^d)$ for $1 \leq j \leq 6$ and such that
\be \label{def:phi1}
\phi(r) = \begin{dcases*} r^2/2 & for $r \leq 1$ \\ \mbox{const.} & for $r \geq 10$ \end{dcases*} \quad \mbox{and} \quad \mbox{$\phi''(r) \leq 1$ for $r \geq 0$},
\ee
 For $R > 0$ given, we define the rescaled  function $\phi_R : \R^d \to \R$ by setting
\be
\phi_R(r) := R^2 \phi \left ( \frac{r}{R} \right ). 
\ee
We readily verify the inequalities
\be \label{ineq:phi1}
1 - \phi_R''(r) \geq 0, \quad 1- \frac{\phi_R'(r)}{r} \geq 0, \quad d - \DD \phi_R(r) \geq 0 \quad \mbox{for all $r \geq 0$}.
\ee
Indeed, this first inequality follows from $\phi_R''(r) = \phi''(r/R) \leq 1$. We obtain the second inequality by integrating the first inequality on $[0,r]$ and using that $\phi_R'(0) = 0$. Finally, we find that $d - \DD \phi_R(r) = 1-\phi_R''(r) + (d-1)\{1 - \frac{1}{r} \phi_R'(r)\} \geq 0$ holds thanks to the first two inequalities in \eqref{ineq:phi1}.  

For later use, we record the following properties of $\phi_R$, which can be easily checked:
\be \label{eq:phi2}
\left \{
\begin{aligned}
& \nabla \phi_R(r) =   R \phi' \left ( \frac{r}{R} \right ) \frac{x}{|x|} = \begin{dcases*} x & for $r \leq R$ \\ 0 & for $r \geq 10R$ \end{dcases*} ; \\ 
& \| \nabla^j \phi_R \|_{L^\infty} \lesssim R^{2-j} \quad \mbox{for $0 \leq j \leq 6$} \, ; \\
& \mathrm{supp} \, ( \nabla^j \phi_R ) \subset \begin{dcases*} \left \{ |x| \leq 10 R \right \} & for $j=1,2$ \\ \left \{ R \leq |x| \leq 10 R \right \} & for $3 \leq j \leq 6$ \end{dcases*} .
\end{aligned}
\right .
\ee
For $u \in H^2(\R^d)$, we define the {\bf localized virial} of $u$ to be the quantity
\be
\MM_{\phi_R}[u] := \left \langle u, -i( \nabla \phi_R \cdot \nabla + \nabla \cdot \nabla \phi_R) u \right \rangle  =  2 \, \mathrm{Im} \int_{\R^d} \overline{u} \nabla \phi_R \cdot \nabla u,
\ee
where the last equality above follows from a simple integration by parts. In fact, we shall use both expressions depending on the situation. By the Cauchy-Schwarz inequality, we have $\left | \MM_{\phi_R}[u] \right | \lesssim R \| u \|_{L^2} \| \nabla u \|_{L^2}$. In particular, the localized virial $\MM_{\phi_R}[u]$ is well-defined for $u \in H^2(\R^d)$. 

\begin{lemma}[Time Evolution of $\MM_R$] \label{lem:MR}
Let $d \geq 2$ and $R>0$. Suppose that $u \in C([0,T); H^2(\R^d))$ is a radial solution of \eqref{eq:bNLS}. Then, for any $t \in [0,T)$, we have the differential inequality
\begin{align*}
\frac{d}{dt} \MM_{\phi_R}[u(t)] & \leq 4 d \sigma  E[u_0] - ( 2 d \sigma - 8) \| \DD u(t) \|_{L^2}^2 - (2d \sigma - 4) \mu \| \nabla u(t) \|_{L^2}^2 + X_\mu[u(t)]  \\
& \quad  + \cO  \left ( R^{-4} +  R^{-2}  \| \nabla u(t) \|_{L^2}^2 + R^{-\sigma(d-1)} \| \nabla u(t) \|_{L^2}^\sigma  + |\mu| R^{-2}  \right ),
\end{align*}
where
$$
X_\mu[u] \lesssim \begin{dcases*} 0 & for $\mu \geq 0$, \\  |\mu| \| \nabla u(t) \|_{L^2}^2  & for $\mu < 0$. \end{dcases*} 
$$
\end{lemma}

\begin{remarks*}
{\em 1.~For non-radial solutions $u \in C([0,T); H^2(\R^d))$ and any $d \geq 1$ and $\sigma > 0$, the above differential inequality also holds formally true except for the error term $\cO( \ldots)$, whose bound crucially relies on the radiality of $u(t,r)$ and the condition  $d \geq 2$.

2.~Localized virial identities for biharmonic NLS have already appeared  in \cite{PaSh-10,MiXuZh-09,Pa-07, Pa-09}. However, the point here is that we show by a careful analysis that certain terms can be shown to have a certain sign, which will be essential for proving blowup theorems based on $\MM_R[u]$.} 
\end{remarks*}

\begin{proof}
We split the proof of Lemma \ref{lem:MR} into the following steps.

\medskip
{\bf Step 1 (Preliminaries and Commutator Identities).} First, we recall that 
$$
\MM_R[u(t)] = \langle u(t), \Gamma_{\phi_R} u(t) \rangle \quad  \mbox{with} \ \ \Gamma_{\phi_R} := -i \left ( \nabla \phi_R \cdot \nabla + \nabla \cdot \nabla \phi_R \right ).
$$
By taking the time derivative and using that $i\partial_t u$ is given by \eqref{eq:bNLS}, we deduce 
\be
\frac{d}{dt} \MM_{\phi_R}[u(t)] =  \AA^{(1)}[u(t)] + \AA^{(2)}[u(t)] + \BB[u(t)] 
\ee
with
$$
 \AA_R^{(1)}[u] := \left \langle u(t), [\DD^2, i \Gamma_{\phi_R}] u(t) \right \rangle, \quad \AA_R^{(2)}[u] := \left \langle u(t), [-\mu \DD, i \Gamma_{\phi_R}] u(t) \right \rangle  ,
$$ 
$$
\BB_R[u] := \left \langle u(t), [-|u|^{2 \sigma}, i \Gamma_{\phi_R}] u(t) \right \rangle.
$$
Since $\DD^2 u \in H^{-2}(\R^d)$ and $\Gamma_{\phi_R}  u \in H^{1}(\R^d)$ in general, we note that the term $\AA^{(1)}[u]$ is not well-defined for $u \in H^2(\R^d)$. Therefore, the following calculations require some higher regularity of $u(t)$; e.\,g., it suffices to assume that $u \in H^{3}(\R^d)$ holds. The claimed identities and inequality then follow by an approximation argument and passing to limits. (For instance, we could employ a Yosida type approximation  with $u_\eps = (-\eps \DD + 1)^{-1} u$ and pass to the limit $\eps \to 0^+$.) We omit the details of such a standard procedure.

As a further preliminary step, we collect some commutator identities that will come in handy below. First, we observe that
\be \label{eq:comm1}
[\DD^2, i \Gamma_{\phi_R}]  = \DD [\DD, i \Gamma_{\phi_R}] + [\DD, i \Gamma_{\phi_R}] \DD  =  2 \partial_k [\DD, i \Gamma_{\phi_R}] \partial_k +  [\partial_k, [\partial_k, [\DD, i \Gamma_{\phi_R}]]].
\ee
Note that we used the fact that $\DD A + A \DD = 2 \partial_k A \partial_k + [\partial_k, [\partial_k, A]]$ for an operator $A$. Next, a calculation yields the known commutator formula
\be \label{eq:comm2}
[\DD, i \Gamma_{\phi_R}] = [ \DD, \nabla \phi_R \cdot \nabla + \nabla \phi_R \cdot \nabla ] =  4 \partial_k (\partial^{2}_{kl} \phi_R) \partial_l + \DD^2 \phi_R. 
\ee
If we plug this back into \eqref{eq:comm1}, we obtain the identity
\be \label{eq:comm3}
[\DD^2, i \Gamma_{\phi_R}] = 8 \partial_{kl}^2 (\pt^{2}_{lm} \phi_R) \pt^{2}_{mk} + 4 \pt_k (\pt^2_{kl} \DD \phi_R) \pt_l + 2 \partial_k (\DD^2 \phi_R) \pt_k  + \DD^3 \phi_R.
\ee
We are now ready to divide the analysis of the terms $\AA^{(1)}[u]$, $\AA^{(2)}[u]$, and $\BB[u]$ into the following steps.

\medskip
{\bf Step 2 (Dispersive Parts $\AA^{(1)}_R$ and $\AA^{(2)}_R$).}  We start by recalling that the Hessian of sufficiently regular and radial function $f : \R^d \to \C$ is given by
\be \label{eq:Hessian}
\pt^2_{kl} f = \left ( \delta_{kl} - \frac{x_k x_l}{r^2} \right ) \frac{\pt_r f}{r} + \frac{x_k x_l}{r^2} \pt_r^2 f.
\ee
Applying this to  $\phi_R(r)$ and $u(t,r)$, a calculation combined with integration by parts yields that
\begin{align*}
8 \left \langle u, \partial_{kl}^2 (\pt^{2}_{lm} \phi_R) \pt^{2}_{mk} u \right \rangle & = 8 \int_{\R^d} ( \pt_{kl}^2 \overline{u})  ( \pt^2_{lm} \phi_R ) ( \pt^2_{mk} u ) \\
&  = 8 \int_{\R^d} \Big( \pt_r^2 \phi_R \, |\pt_r^2 u|^2 + \frac{d - 1}{r^2} \, \frac{\pt_r \phi_R}{r} \, |\pt_r u|^2 \Big) \\
& = 8 \int_{\R^d} |\DD u|^2 - \left (1 - \pt_r^2 \phi_R \right ) |\pt_r^2 u|^2 - \left ( 1- \frac{\pt_r \phi_R}{r} \right ) \frac{d-1}{r^2} |\pt_r u|^2 .
\end{align*}
Here we also used the identity $\int_{\R^d} |\Delta u|^2  = \int_{\R^d} \left \{ |\pt_r^2 u|^2 + \frac{d - 1}{r^2} \, |\pt_r u|^2 \right \}$ for radial $u \in H^2(\R^d)$, which follows from integration by parts in $r=|x|$. In view of the inequalities \eqref{ineq:phi1}, we deduce the bound 
\be \label{ineq:MR1}
8 \left \langle u, \partial_{kl}^2 (\pt^{2}_{lm} \phi_R) \pt^{2}_{mk} u \right \rangle \leq 8 \int_{\R^d} |\DD u|^2 .
\ee
Furthermore, straightforward arguments yield that
\be
\begin{aligned} \label{ineq:MR2}
& \left |  \left \langle u,  \pt_k (\pt^2_{kl} \DD \phi_R) \pt_l  u \right \rangle \right | \lesssim \| \pt^2_{kl} {\DD \phi_R} \|_{L^\infty} \| \nabla u \|_{L^2}^2 \lesssim R^{-2} \| \nabla u \|_{L^2}^2,\\
& \left |  \left \langle u, \partial_k (\DD^2 \phi_R) \pt_k u \right \rangle \right  | \lesssim \| \DD^2 \phi_R \|_{L^\infty} \| \nabla u \|_{L^2}^2 \lesssim R^{-2} \| \nabla u \|_{L^2}^2, \\
&  \left | \left \langle u, \DD^3 \phi_R u \right \rangle �\right | \lesssim \| \DD^3 \phi_R \|_{L^\infty} \| u \|_{L^2}^2 \lesssim R^{-4} \| u \|_{L^2}^2.  
\end{aligned}
\ee
By combining the bounds in \eqref{ineq:MR1} and \eqref{ineq:MR2}, we conclude that
\be \label{eq:A1}
\AA^{(1)}_R[u(t)] \leq 8 \int_{\R^d} |\Delta u(t)|^2 + \cO \left ( R^{-4} + R^{-2} \| \nabla u(t) \|_{L^2}^2 \right ).
\ee
Next, let us turn to $\AA^{(2)}_R[u]$. Here we use \eqref{eq:Hessian} and \eqref{eq:comm2} and find by calculation that
\begin{align*}
\AA_R^{(2)}[u] & =  4 \mu \int_{\R^d} ( \pt_k \overline{u} ) (\pt_{kl}^2 \phi_R) ( \pt_l u) - \mu \int_{\R^d} (\DD^2 \phi_R) |u|^2 \\
& = 4 \mu \int_{\R^d} (\pt_r^2 \phi_R) |\pt_r u|^2 - \mu \int_{\R^d} (\DD^2 \phi_R) |u|^2 \\
& = 4 \mu \int_{\R^d} |\nabla u|^2 + X_\mu[u(t)] - \mu \int_{\R^d} (\DD^2 \phi_R) |u|^2,
\end{align*}
with
\be
X_\mu[u] = - 4 \mu \int_{\R^d} (1- \pt_r^2 \phi_R) |\pt_r u|^2 .
\ee 
From \eqref{ineq:phi1} and \eqref{eq:phi2} we recall that $1-\pt_r^2 \phi_R \geq 0$ and $\|1-\pt_r^2 \phi_R \|_{L^\infty} \lesssim 1$. Hence,
\be
X_\mu[u] \lesssim \begin{dcases*} 0 & for $\mu \geq 0$, \\ |\mu| \| \nabla u \|_{L^2}^2 & for $\mu < 0$. \end{dcases*} 
\ee
Since $\| \DD^2 \phi_R \|_{L^\infty} \lesssim R^{-2}$, we finally obtain
\be \label{eq:A2}
\AA^{(2)}_R[u(t)] = 4 \mu \int_{\R^d} |\nabla u(t)|^2 + X_\mu[u(t)] + \cO \left ( |\mu| R^{-2} \right ) .
\ee

\medskip
{\bf Step 3 (Nonlinearity Term $\BB_R[u]$ and Conclusion).}  Here we note that integration by parts yields 
\begin{align*}
\BB_R[u] & = - \left \langle u, [|u|^{2 \sigma},  \nabla \phi_R \cdot \nabla + \nabla \cdot \nabla \phi_R] u \right \rangle = 2 \int_{\R^d} |u|^2 \nabla \phi_R \cdot  \nabla (|u|^{2 \sigma}) \\
& =  - \frac{2 \sigma}{\sigma+1} \int_{\R^d} (\DD \phi_R) |u|^{2 \sigma +2}, 
\end{align*}
where we also made use of the identity $\nabla (|u|^{2 \sigma+2}) = \frac{\sigma+1}{\sigma} \nabla (|u|^{2 \sigma}) |u|^2$. Since $\phi_R(r) = r^2/2$ for $r \leq R$ and hence $\DD \phi_R(r) - d \equiv 0$ for $r \leq R$, we obtain
\begin{align*}
\BB_R[u] & = - \frac{2 \sigma d}{\sigma+1} \int_{\R^d} |u|^{2 \sigma+2} - \frac{2 \sigma}{\sigma+1} \int_{|x| \geq R} ( \DD \phi_R -d ) |u|^{2 \sigma+2} \\
& = - \frac{2 \sigma d}{\sigma+1} \int_{\R^d} |u|^{2 \sigma+2} + \cO \left ( R^{-\sigma(d-1)} \| \nabla u \|_{L^2}^{\sigma} \right ), 
\end{align*}
where the last step follows from $\| \DD \phi_R - d \|_{L^\infty} \lesssim 1$ and applying the Strauss inequality, which gives us
$$
 \int_{|x| \geq R} |u|^{2 \sigma + 2}  \lesssim\|u\|_{L^2}^2   \|u\|_{L^\infty(|x| \geq R)}^{2 \sigma} \lesssim R^{- \sigma (d-1)} \|u\|_{L^2}^{2 + \sigma} \|\nabla u\|_{L^2}^\sigma.
$$
Finally, we combine \eqref{eq:A1} and \eqref{eq:A2} with the estimate for $\BB_{R}[u]$ to deduce that 
\begin{align*}
\frac{d}{dt} \MM_R[u(t)] & \leq 8 \int_{\R^d} |\Delta u(t)|^2 + 4 \mu \int_{\R^d} |\nabla u|^2 - \frac{2 \sigma d}{\sigma+1} \int_{\R^d} |u|^{2 \sigma+2} + X_\mu[u(t)]\\
& \quad + \cO \left ( R^{-4} + R^{-2}  \| \nabla u(t) \|_{L^2}^2 + R^{-\sigma(d-1)} \| \nabla u(t) \|_{L^2}^\sigma  +  |\mu| R^{-2} \right ) \\
& = 4 d \sigma E[u_0] -  (2 d \sigma -8) \| \DD u(t) \|_{L^2}^2 - (2d \sigma -4) \mu \| \nabla u(t) \|_{L^2}^2 + X_\mu[u(t)] \\
&  \quad + \cO \left (  R^{-4} +  R^{-2}  \| \nabla u(t) \|_{L^2}^2 + R^{-\sigma(d-1)} \| \nabla u(t) \|_{L^2}^\sigma +  |\mu| R^{-2} \right ),
\end{align*}
where we also used the conservation of energy $E[u(t)]=E[u_0]$. 

This completes the  proof of Lemma \ref{lem:MR}.
 \end{proof}

\section{Existence of Blowup for Mass-Supercritical Case} \label{sec:blowup_super}

In this section, we will prove Theorems \ref{thm:blowup} and \ref{thm:energy}. With Lemma \ref{lem:MR} at hand, we can follow a strategy that has been introduced by Ogawa and Tsutsumi to show blowup for radial (infinite-variance) solutions for NLS; see also \cite{SuSu-99} for a review on this method as well as \cite{KeMe-06,KiVi-10} for energy-critical NLS. Although the proofs of Theorems \ref{thm:blowup} and \ref{thm:energy} are very similar, we give them separately for the sake of clarity.

\subsection{Proof of Theorem \ref{thm:blowup}}
Let us assume that $d \geq 2$, $\mu \in \R$, and $0 < s_c < 2$ with $\sigma \leq 4$. Suppose that $u_0 \in H^2(\R^d)$ is radial and let $u \in C^0([0;T);H^2(\R^d)))$ be the solution of \eqref{eq:bNLS}.

For $R > 0$, we let $\phi_R(r)= \phi(r/R)$ be the radial cutoff function introduced in Section \ref{sec:locvirial} above. For notational convenience, we write
$$
\MM_{R}[u(t)] \equiv \MM_{\phi_R}[u(t)]
$$
to denote the localized virial defined in Section \ref{sec:locvirial} above. Furthermore, we define the number 
\be
\delta := d \sigma -4
\ee for notational convenience. We split the rest of the proof according to the following three cases, which clearly cover the assertions (i) and (ii) in Theorem \ref{thm:blowup}.

\medskip
{\bf Case 1: $\mu \geq 0$ and $E[u_0] < 0$}.  From Lemma \ref{lem:MR}, we deduce that
\begin{align*}
\frac{d}{dt} \MM_R[u(t)] & \leq 4 d \sigma E[u_0] - 2 \delta \| \DD u(t) \|_{L^2}^2  \\
& \quad + \cO \left ( R^{-4} + R^{-2} \| \DD u(t) \|_{L^2} + R^{-\sigma(d-1)} \| \DD u(t) \|_{L^2}^{\sigma/2} + |\mu| R^{-2} \right ), 
\end{align*}
where we also used that $\| \nabla u(t) \|_{L^2} \leq C(u_0) \| \DD u(t) \|_{L^2}^{1/2}$. Since $\sigma \leq 4$ and $E[u_0] < 0$ by assumption, we can choose $R > 0$ sufficiently large such that
\be \label{ineq:virialmaster}
\frac{d}{dt} \MM_R[u(t)] \leq 2 d \sigma E[u_0] - \delta \| \DD u(t) \|_{L^2}^2 \quad \mbox{for $t \in [0,T)$}. 
\ee
We are now ready to argue by contradiction as follows. Suppose that $T=+\infty$ holds. From \eqref{ineq:virialmaster} we conclude that $\MM_R[u(t)] \leq 0$ for all $t \geq t_1$ with some sufficiently large time $t_1 \geq 0$. In particular, we have $\mathcal{M}_{R}[u(t_1)] \leq 0$. Hence, by integrating \eqref{ineq:virialmaster2} on $[t_1,t]$ with $t>t_1$ and using that $E[u_0] \leq 0$, we get
$$
\MM_{R}[u(t)] \leq -\delta \int_{t_1}^t \|\Delta u(s) \|_{L^2}^2 \, ds \leq 0.
$$
Next, by the Cauchy--Schwarz inequality,
$$
| \MM_{R}[u(t)] | \lesssim \| \nabla \phi_R \|_{L^\infty} \| u (t) \|_{L^2} \| \nabla u(t)\|_{L^2} \leq C(u_0) R \| \Delta u(t) \|_{L^2}^{\frac 1 2}.
$$
Thus we find
$$
\MM_{R}[u(t)] \leq - A  \int_{t_1}^t  \left |  \MM_R[u(s)] \right |^4 \, ds  \quad \mbox{with $A := C(\delta,R) > 0$}.
$$
Let us  define $z(t) := \int_{t_1}^t \left | \mathcal{M}_{R}[u(s)] \right |^4 \,ds$ for $t \geq t_1$ and fix some time $t_2 > t_1$. Clearly, the function $z(t)$ is strictly increasing and nonnegative. Moreover, we have $ \mathcal{M}_{R}[u(t)] = z'(t) \geq A^4 z(t)^4$. Hence, if we integrate this differential inequality on $[t_2, t]$, we obtain
$$
\mathcal{M}_{R}[u(t)] \leq -A z(t) \leq \frac{-A z(t_2)}{\left (1-3 A^4 z(t_2)^3 (t-t_2) \right )^{\frac 1 3}} \quad \mbox{for all $t > t_2$}.
$$
But this shows that $\mathcal{M}_{R}[u(t)] \to -\infty$ as $t \to t_*$ for some finite time $t_* < +\infty$. Therefore, the solution $u(t)$ cannot exist for all $t \geq 0$. By the blowup alternative for the energy-subcritical case $s_c < 2$, this completes the proof of Theorem \ref{thm:blowup} for $\mu \geq 0$ and $E[u_0] < 0$.

\medskip
{\bf Case 2: $\mu < 0$.} We apply Lemma \ref{lem:MR} to find
\begin{align*}
\frac{d}{dt} \MM_R[u(t)] & \leq 4 d \sigma E[u_0] - 2 \delta \| \DD u(t) \|_{L^2}^2  + A |\mu| \| \nabla u(t) \|_{L^2}^2 \\
& \quad + \cO \left ( R^{-4} + R^{-2} \| \DD u(t) \|_{L^2} + R^{-\sigma(d-1)} \| \DD u(t) \|_{L^2}^{\sigma/2} + |\mu| R^{-2} \right )
\end{align*}
with some universal constant $A > 0$. Now we use $\| \nabla u \|_{L^2}^2 \leq \frac{1}{2 \eta} \| u \|_{L^2}^2 + \frac{\eta}{2} \| \DD u \|_{L^2}$ with $\eta = 2 \delta /(A |\mu|)$, which yields 
\begin{align*}
\frac{d}{dt} \MM_R[u(t)] & \leq 4 d \sigma E[u_0] +  \frac{A^2 \mu^2}{4 \delta} M[u_0] -  \delta \| \DD u(t) \|_{L^2}^2   \\
& \quad + \cO \left ( R^{-4} + R^{-2} \| \DD u(t) \|_{L^2} + R^{-\sigma(d-1)} \| \DD u(t) \|_{L^2}^{\sigma/2} + |\mu| R^{-2} \right ) \\
& = 4 d \sigma \left ( E[u_0]  + \varkappa \mu^2 M[u_0] \right ) -  \delta \| \DD u(t) \|_{L^2}^2  \\
& \quad + \cO \left ( R^{-4} + R^{-2} \| \DD u(t) \|_{L^2} + R^{-\sigma(d-1)} \| \DD u(t) \|_{L^2}^{\sigma/2} + |\mu| R^{-2} \right ) ,
\end{align*}
where we have set $\varkappa := A^2 / (16  \delta d \sigma)$. Thus if we assume that $E[u_0] + \varkappa \mu^2 M[u_0] < 0$ and choose $R > 0$ sufficiently large, we deduce
$$ 
  \frac{d}{dt} \MM_R[u(t)] \leq e - \delta' \| \DD u(t) \|_{L^2}^2 \quad \mbox{for all $t \in [0,T)$}
$$
with some constants $e > 0$ and $\delta' > 0$. If we now use the arguments presented following \eqref{ineq:virialmaster} above, we deduce that $u(t)$ must blowup in finite time.

\medskip
{\bf Case 3: $\mu=0$ and $E(u_0) \geq 0$}. Suppose that $\mu=0$ holds and assume that $E(u_0) \geq 0$ satisfies the conditions
\be  \label{ineq:ass1}
E[u_0]^{s_c} M[u_0]^{2-s_c} < E[Q]^{s_c} M[Q]^{2-s_c} =: \Lambda[Q],
\ee
\be \label{ineq:ass2}
\| \DD u_0 \|_{L^2}^{s_c} \| u_0 \|_{L^2}^{2-s_c} > \| \DD Q \|_{L^2}^{s_c} \| Q \|_{L^2}^{2-s_c}.
\ee
Next, by using energy conservation, we notice the lower bound
\be \label{ineq:E0lower}
E[u_0] = \frac{1}{2} \| \DD u(t) \|_{L^2}^2 - \frac{1}{2 \sigma +2} \| u(t) \|_{L^{2 \sigma+2}}^{2 \sigma + 2} \geq F \left ( \| \DD u(t) \|_{L^2} \right ),
\ee
where the last inequality follows from $L^2$-mass conservation $M[u(t)] = M[u_0]$ and the interpolation inequality \eqref{ineq:GNS} with the function $F :[0,�\infty) \to \R$ defined as
\be
F(y) := \frac{1}{2} y^2 - \frac{C_{d, \sigma}}{2 \sigma +2} M[u_0]^{\frac{\sigma}{2}( 2- s_c)} y^{2 + \sigma s_c} .
\ee 
Here $C_{d, \sigma} > 0$ denotes the optimal constant for inequality \eqref{ineq:GNS}. It is straightforward to check that $F(y)$ has a unique global maximum  attained at
\be \label{eq:ymax}
y_{\max} = (K_{d, \sigma})^{\frac{1}{s_c}} M[u_0]^{- \frac{2-s_c}{2 s_c}}  \quad \mbox{with} \quad K_{d,\sigma}=\left ( \frac{4(\sigma+1)}{d \sigma C_{d, \sigma}} \right )^{\frac{1}{\sigma}},
\ee
and 
\be
F(y_{\max}) = \frac{s_c}{d} y_{\max}^2 .
\ee
On the other hand, by Pohozaev identities, we obtain 
$$
K_{d, \sigma} = \| \DD Q \|_{L^2}^{s_c} \| Q \|_{L^2}^{2-s_c} = \left ( \frac{s_c}{d} \right )^{-\frac{s_c}{2}} \Lambda[Q]^{\frac 1 2}. 
$$
Using this, we conclude that the conditions \eqref{ineq:ass1}--\eqref{ineq:ass2} imply that 
$$
E[u_0] < F(y_{\max}) \quad \mbox{and} \quad \| \DD u_0 \|_{L^2} > y_{\max}.
$$ 
In view of \eqref{ineq:E0lower} and by continuity in time, we deduce that
\be \label{ineq:DDulower}
\| \DD u(t) \|_{L^2} > y_{\max} \quad \mbox{for all $t \in [0,T)$},
\ee
since otherwise there exists $t_* \in (0,T)$ such that $\| \DD u(t_*) \|_{L^2} = y_{\max}$, which contradicts \eqref{ineq:E0lower} and $E[u_0] < F(y_{\max})$. Next,  we choose $\eta > 0$ sufficiently small such that
$$
E[u_0]^{s_c} M[u_0]^{2-s_c}  \leq (1-\eta)^{s_c} \Lambda[Q].
$$
Using \eqref{ineq:DDulower}, an elementary calculation yields that
$$
2\delta (1-\eta)  \| \DD u(t) \|_{L^2}^2 \geq 4 d \sigma E[u_0] \quad \mbox{for all $t \in [0,T)$},
$$
where we recall that $\delta = d \sigma -4$. Thus from Lemma \ref{lem:MR} and the previous discussion we obtain from inequality \eqref{ineq:virialmaster} the upper bound
\be
\begin{aligned}
& \frac{d}{dt} \MM_R[u(t)] \leq 4 d \sigma E[u_0] - 2 \delta \| \DD u(t) \|_{L^2}^2 \\
& \phantom{\frac{d}{dt} \MM_R[u(t)] } \quad + \cO \left ( R^{-4} + R^{-2} \| \DD u(t) \|_{L^2} + R^{-\sigma(d-1)} \| \DD u(t) \|_{L^2}^{\sigma/2} \right ) \\
& \phantom{\frac{d}{dt} \MM_R[u(t)] } \leq   - \left ( \delta \eta + o_R(1) \right ) \| \DD u(t) \|_{L^2}^2 + o_R(1),
\end{aligned}
\ee
with $o_R(1) \to 0$ as $R \to \infty$ uniformly in $t$. Thus by choosing $R > 0$ sufficiently large and using the uniform lower bound \eqref{ineq:DDulower}, we conclude 
\be \label{ineq:virialmaster2}
\frac{d}{dt} \MM_R[u(t)] \leq -\frac{\delta \eta}{2}  \| \DD u(t) \|_{L^2}^2 \quad \mbox{for all $t \in [0,T)$}.
\ee
We are now ready to argue by contradiction as follows. Suppose that $T=+\infty$ holds. Using the uniform lower bound $\| \DD u(t) \|_{L^2} > y_{\max} > 0$ for all $t \geq 0$ and integrating \eqref{ineq:virialmaster2}, we conclude that $\MM_R[u(t)] \leq 0$ for all $t \geq t_1$ with some sufficiently large time $t_1 \geq 0$. In particular, we have $\mathcal{M}_{R}[u(t_1)] \leq 0$. Hence, by integrating \eqref{ineq:virialmaster2} on $[t_1,t]$ with $t>t_1$, we get
$$
\MM_{R}[u(t)] \leq -\frac{\delta \eta}{2} \int_{t_1}^t \|\Delta u(s) \|_{L^2}^2 \, ds \leq 0.
$$
As before, this integral inequality implies that $u(t)$ blows up in finite time. 

The proof of Theorem \ref{thm:blowup} is now complete. \hfill $\blacksquare$

\subsection{Proof of Theorem \ref{thm:energy}}
Let $d \geq 5$ and $\sigma = \frac{4}{d-4}$, i.\,e., we assume that $s_c=2$ holds. Suppose $u_0 \in H^2(\R^d)$ is radial and let $u \in C^0([0,T); H^2(\R^d))$ denote the corresponding solution of \eqref{eq:bNLS}. Since the proof of Theorem \ref{thm:energy} is very similar to the one given for Theorem \ref{thm:blowup} above, we only discuss the following case and leave the remaining (simpler) cases to the reader.

Let us suppose that $\mu= 0$ holds and assume $u_0 \in H^2(\R^d)$ is radial with
\be \label{ineq:Wass}
0 \leq E[u_0] < E[W] \quad \mbox{and} \quad \| \DD u_0 \|_{L^2} > \| \DD W \|_{L^2}.
\ee
For notational convenience, we set $p= \frac{2d}{d-4}$. By energy conservation and Sobolev's inequality, we have the lower bound
\be
E[u_0] = \frac{1}{2} \| \DD u(t) \|_{L^2}^2 - \frac{1}{p} \| u(t) \|_{L^p}^p \geq F \left (\| \DD u(t)\|_{L^2} \right ),
\ee
where the function $F : [0,\infty) \to \R$ is given by
\be
F(y) := \frac{1}{2} y^2 - \frac{C_d^p}{p} y^{p}.
\ee
Recall that $C_d > 0$ denotes the optimal constant for inequality \eqref{ineq:sob}. Again, we notice that $F(y)$ has a unique global maximum given by
$$
F(y_{\max}) = \frac{2}{d} y_{\max}^2 \quad \mbox{with} \quad y_{\max} =  \left ( \frac{1}{C_d} \right )^{\frac{d}{4}}.
$$
On the other hand, by the Pohozaev identities \eqref{eq:Wpoho},
$$
\| \DD W \|_{L^2} = y_{\max} \quad \mbox{and} \quad F(y_{\max}) = E[W].
$$
Thus from \eqref{ineq:Wass} we infer that
$$
E[u_0] < F(y_{\max}) \quad \mbox{and} \quad \| \DD u_0 \|_{L^2} > \| \DD W \|_{L^2}.
$$
By a simple continuity argument, we deduce that $\| \DD u(t) \|_{L^2} > \| \DD W \|_{L^2}$ for all $t \in [0,T)$, as in the proof of Theorem \ref{thm:blowup}. Next, from Lemma \ref{lem:MR} we obtain
\be \label{ineq:virial_energy}
\begin{aligned}
\frac{d}{dt} \MM_R[u(t)] & \leq \frac{16d}{d-4} \left (  E[u_0] - \frac{2}{d} \| \DD u(t) \|_{L^2}^2 \right ) \\
& \quad +  \cO \left ( R^{-4} + R^{-2} \| \DD u(t) \|_{L^2} + R^{-\frac{4(d-1)}{d-4}} \| \DD u(t) \|_{L^2}^{\frac{2}{d-4}} \right ) .
\end{aligned}
\ee
Now we choose $\eta > 0$ sufficiently small such that
$$
E[u_0] \leq (1-\eta) E[W] .
$$
Since $\| \DD u(t) \|_{L^2} > \| \DD W \|_{L^2}$ and $E[W] = \frac{2}{d} \| \DD W \|_{L^2}^2$, we deduce
$$
(1-\eta) \frac{2}{d} \| \DD u(t) \|_{L^2}^2 \geq E[u_0] \quad \mbox{for all $t \in [0,T)$}.
$$
Going back to \eqref{ineq:virial_energy} and choosing $R >0$ sufficiently large, we conclude
\be \label{ineq:final_virial}
\frac{d}{dt} \MM_R[u(t)]  \leq - \frac{16}{d-4}  \eta \| \DD u(t) \|_{L^2}^2  \quad \mbox{for all $t \in [0,T)$},
\ee
where we also made use of the uniform lower bound $\| \DD u(t) \|_{L^2}> \| \DD W \|_{L^2}$ to absorb the error term $\cO(R^{-4})$. With estimate \eqref{ineq:final_virial} at hand, we can now conclude that $u(t)$ cannot exist for all times $t \geq 0$, in the same fashion as we did with  \eqref{ineq:virialmaster2} in the proof of Theorem \ref{thm:blowup} above.

This completes the proof of Theorem \ref{thm:energy}. \hfill $\blacksquare$

\section{Localized Riesz Bivariance and Estimates} \label{sec:locrieszvar}

Let $\phi : \R^d \to \R$ be a radial function as in Section \ref{sec:locvirial} above. In addition, we suppose that
\be \label{eq:phi_Riesz}
\left \{
\begin{aligned}
& \phi(r) \geq 0 \quad \mbox{for $r \geq 0$} \quad \mbox{and} \quad \phi(r) \equiv 0 \quad \mbox{for $r \geq 10$}, \\
& \nabla^j \sqrt{\phi} \in L^\infty(\R^d) \quad \mbox{for $0 \leq j \leq 6$}.
 \end{aligned}
 \right .
\ee
For details on how to choose such a function $\phi(r)$, we refer to Appendix \ref{sec:cutoff}. For $R > 0$, we define the rescaled function
$$
\phi_R(r) := R^2 \phi \left ( \frac{r}{R} \right )
$$
Now we introduce another radial cutoff function $\psi_R : \R^d \to \R$ that is given by
\be \label{def:psiR}
\psi_R(r) := \int_0^r \sqrt{2 \phi_R(s)} \, d s. 
\ee
It is elementary to check that
\be \label{eq:psi1}
\begin{aligned}
& \psi_R(r) = \begin{dcases*} r^2/2 & for $r \leq R$ \\ \mbox{const.} & for $r \geq 10R$ \end{dcases*}, \quad \nabla \psi_R(r) = \begin{dcases*} x & for $r \leq R$ \\ 0 & for $r \geq 10R$ \end{dcases*} ; \\
& \| \nabla^j \psi_R \|_{L^\infty} \lesssim R^{2-j} \quad \mbox{for $0 \leq j \leq 6$} \, ; \\
& \mathrm{supp} \, ( \nabla^j \psi_R ) \subset \begin{dcases*} \left \{ |x| \leq 10 R \right \} & for $j=1,2$ \\ \left \{ R \leq |x| \leq 10 R \right \} & for $3 \leq j \leq 6$ \end{dcases*} .
\end{aligned}
\ee
Furthermore by differentiating $\phi_R(r) = \frac{1}{2} |\nabla \psi_R(r)|^2= \frac{1}{2} (\partial_k \psi_R)(\partial_k \psi_R)$, we deduce that
\be \label{eq:mastercutoff}
\partial_l \phi_R = (\pt^{2}_{kl} \psi_R)( \partial_k \psi_R),
\ee  
for $l=1, \ldots, d$. This identity will be used below.

For the rest of this section, we assume that $d \geq 3$ holds. We define the {\bf localized Riesz bivariance} by setting
\be
\VV_{\psi_R}[u] := \left \langle u, \nabla \psi_R \cdot (-\DD)^{-1} \nabla \psi_R u \right \rangle = \left  \langle \partial_k \psi_R  u, (-\DD)^{-1} \partial_k \psi_R u \right \rangle .
\ee 
Using that $(-\DD)^{-1} = |\nabla|^{-2}$ and by Plancherel's theorem, we discover that
\be \label{id:VirPos}
\VV_{\psi_R}[u] = \| |\nabla|^{-1} (\nabla \psi_R u) \|_{L^2}^2 = \int_{\R^d} |\xi|^{-2} \big | \widehat{( \nabla \psi_R u)}(\xi) \big |^2 \, d \xi.
\ee 
Clearly $\VV_{\psi_R}[u] \geq 0$ is nonnegative and finite for $u \in H^2(\R^d)$, since we have
\be \label{ineq:HardyVR}
\begin{aligned}
\VV_{\psi_R}[u] & = \| |\nabla|^{-1} (\nabla \psi_R u) \|_{L^2}^2 \leq C \| |x| (\nabla \psi_R u ) \|_{L^2}^2 \leq C \| |x| \nabla \psi_R \|_{L^\infty}^2 \| u \|_{L^2}^2 \\
                & \lesssim R^4 \| u \|_{L^2}^2,
\end{aligned}
\ee
using the Hardy-type inequality $\| |\nabla|^{-1} f \|_{L^2} \leq  C \| |x| f \|_{L^2}$ valid in dimensions $d \geq 3$. (Notice that $\VV_{\psi_R}[u]$ is already finite if we only assume that $u$ belongs to $L^2(\R^d)$.) 

Suppose now that $u \in C^0([0;T); H^2(\R^d))$ solves \eqref{eq:bNLS} and let $\psi_R$ be as above with $R > 0$ given. For the rest of the section, let us denote the localized Riesz bivariance and the localized virial by
\be
\VV_{R}[u(t)] \equiv \VV_{\psi_R}[u(t)], \quad \MM_R[u(t)] \equiv M_{\phi_R}[u(t)],
\ee
respectively. 

\begin{remark*} {\em
We emphasize that we use the {\em different} cutoff functions $\psi_R$ and $\phi_R$ for $\VV_{R}[u]$ and $\MM_{R}[u]$, respectively, where the relation \eqref{eq:mastercutoff} will be important.}
\end{remark*}

We have the following technical main result.

\begin{lemma}[Time Evolution of $\VV_{R}$] \label{lem:VR} 
Let $d \geq 3$ and suppose $u \in C([0,T); H^2(\R^d))$ is a radial solution of \eqref{eq:bNLS}. Then, for any $t \in [0,T)$, it holds that
$$
\frac{d}{dt} \VV_{R}[u(t)] = 4 \, \MM_{R}[u(t)] + \NN_{R}[u(t)] + \cO \left ( 1 + | \mu| R^2 \right ),
$$
where
$$
 \NN_R[u] = - i \left \langle u(t), [|u|^{2 \sigma}, \pt_k \psi_R (-\DD)^{-1} \pt_k \psi_R] u(t) \right \rangle .
$$
\end{lemma}

\begin{remark*}
{\em In Lemma \ref{lem:NR} below, we will derive estimates that will in particular show that $\NN_R[u]$ is finite for $u \in H^2(\R^d)$.}
\end{remark*}

\begin{proof} For notational convenience, we define the pseudo-differential operator
$$
\Psi_R := \partial_k \psi_R (-\DD)^{-1} \partial_k \psi_R,
$$
which corresponds to a localized version of the Riesz potential $(-\DD)^{-1}$. 

We divide the proof of Lemma \ref{lem:VR} into several steps as follows.

\medskip
{\bf Step 1 (Regularity and Preliminaries).} By using that $i \partial_t u = \DD^2 u -\mu \DD u - |u|^{2 \sigma} u$, a simple computation  yields
\be \label{eq:VR}
\frac{d}{dt} \VV_R[u(t)]  =  \LL_R^{(1)}[u(t)] + \LL_R^{(2)}[u(t)] + \NN_R[u(t)]
\ee
with
\be
\LL_R^{(1)}[u(t)]:= \left \langle u, [\DD^2, i \Psi_R] u \right \rangle, \quad \LL_R^{(2)}[u] :=  - \mu \left \langle u, [\DD, i \Psi_R] u \right \rangle,
\ee 
\be
\NN_R[u] := - \left \langle u, [|u|^{2 \sigma}, i \Psi_R] u \right \rangle.
\ee
Note that all expressions involved here are well-defined due to the smoothing properties of the pseudo-differential operator $\Psi_R$. For instance, since $\DD^2 u \in H^{-2}(\R^d)$, we see that $\LL_R^{(1)}[u]$ is finite provided that $\Psi_R (u)$ belongs to $H^2(\R^d)$. To see this, we first note that, by Sobolev inequalities and the fact that $\nabla \psi_R$ is bounded and  compactly supported,
\be
\| \nabla \psi_R u \|_{L^r} \lesssim C(\psi_R) \| u \|_{H^2}
\ee
for $r \in [1, \infty]$ if $d=3$, $r \in [1, \infty)$ if $d=4$, and $r \in [1, \frac{2d}{d-4}]$ if $d \geq 5$. Thus, by the weak Young inequality, we deduce
\be \label{ineq:weak1}
\| (-\DD)^{-1} (\nabla \psi_R u) \|_{L^q} \lesssim C(\psi_R) \| u \|_{H^2},
\ee
for $q \in (\frac{d}{d-2}, \infty)$ if $3 \leq d \leq 8$ and $q \in (\frac{d}{d-2}, \frac{2d}{d-8})$ if $d \geq 9$. Likewise and using the Mikhlin multiplier theorem, we conclude
\be \label{ineq:weak2}
\| \nabla (-\DD)^{-1} (\nabla \psi_R u) \|_{L^q} \lesssim \| |\nabla|^{-1} (\nabla \psi_R u ) \|_{L^q} \lesssim C(\psi_R) \| u \|_{H^2},
\ee
for $q \in (\frac{d}{d-1}, \infty)$ if $3 \leq d \leq 6$ and $q \in (\frac{d}{d-1}, \frac{2d}{d-6})$ if $d \geq 7$.  We now use \eqref{ineq:weak1} and \eqref{ineq:weak2} to find that
\begin{align*}
 \| \Psi_R(u)  \|_{H^2} & \lesssim \| (\DD \nabla \psi_R) (-\DD)^{-1} ( \nabla \psi_R u ) \|_{L^2} + \| (\nabla^2 \psi_R ) \cdot \nabla (-\DD)^{-1} (\nabla \psi_R u) \|_{L^2} \\
 & \quad + \| \nabla \psi_R \cdot \DD    (-\DD)^{-1} ( \nabla \psi_R u )   \|_{L^2} + \| \nabla \psi_R \cdot (-\DD)^{-1} (\nabla \psi_R u) \|_{L^2}\\
 & \lesssim \| \DD \nabla \psi_R \|_{L^{p_1}} \| (-\DD)^{-1} (\nabla \psi_R u ) \|_{L^{q_1}}  + \| \nabla^2 \psi_R \|_{L^{p_2}} \| |\nabla|^{-1} ( \nabla \psi_R u) \|_{L^{q_2}} \\
 & \quad + \| \nabla \psi_R \|_{L^\infty}^2 \| u \|_{L^2} + \| \nabla \psi_R \|_{L^{p_3}} \| \| (-\DD)^{-1} (\nabla \psi_R u ) \|_{L^{q_3}} \\
 & \lesssim C(\psi_R) \| u \|_{H^2},
\end{align*}
where $1/p_i + 1/q_i = 1/2$ for $i=1,2,3$. We readily verify that $(p_i, q_i)=(4,4)$ for $i=1,2,3$ is an admissible choice when $3 \leq d \leq 4$. For dimensions $d \geq 5$, we can take $(p_i, q_i)=(\infty, 2)$ for $i=1,2,3$.

By following similar arguments as above, we see that the remaining terms in \eqref{eq:VR} are well-defined for $u \in H^2(\R^d)$. We omit the details.

\medskip
{\bf Step 2 (Analysis of $\LL_R^{(1)}$).} We now discuss the term $\LL_{R}^{(1)}$ appearing on the right side in \eqref{eq:VR}. Using that $[A,BC] = [A,B]C + B[A,C]$ and $[\DD^2, (-\DD)^{-1}] = 0$, we note
$$
[\DD^2, i \Psi_R]  = i [\DD^2, \pt_k \psi_R] (-\DD)^{-1} \pt_k \psi_R + i \pt_k \psi_R (-\DD)^{-1} [\DD^2, \pt_k \psi_R] =: i( Z - Z^*)
$$
where we set $Z:= [\DD^2, \pt_k \psi_R] (-\DD)^{-1} \pt_k \psi_R$. Next, by iterating with the identity $[AB, C] = A[B,C] + [A,C] B$, we obtain that
\begin{align*}
[\DD^2, \pt_k \psi_R] & = \DD [ \DD, \pt_k \psi_R] + [\DD, \pt_k \psi_R] \DD = 2 [ \DD, \pt_k \psi_R] \DD + [ \DD, [\DD, \pt_k \psi_R]] \\
& = 2 \left ( \pt_l [ \pt_l, \pt_k \psi_R] + [ \pt_l, \pt_k \psi_R] \pt_l \right ) \DD + [\DD, [\DD, \pt_k \psi_R]] \\
& = 2 \left ( 2 \pt_l [ \pt_l, \pt_k \psi_R] + [ [ \pt_l, \pt_k \psi_R], \pt_l ] \right ) \DD + [\DD, [\DD, \pt_k \psi_R]] \\
& = 4 \pt_l (\pt^2_{kl} \psi_R) \DD - 2 (  \DD \pt_k \psi_R) \DD + [\DD, [\DD, \pt_k \psi_R]].
\end{align*}
We proceed to study the last term on the right side. Here we observe that 
$$
[\DD, \pt_k \psi_R] = \pt_l [ \pt_l, \pt_k \psi_R] + [\pt_l, \pt_k \psi_R] \pt_l = \pt_l (  \pt^2_{kl} \psi_R) + (\pt^2_{kl} \psi_R) \pt_l.
$$
If we apply identity \eqref{eq:comm2} with $\pt_k \psi_R$ instead of $\phi_R$, we find
\begin{align*}
[\DD, [\DD, \pt_k \psi_R]] & = 4 \partial_l ( \partial_{klm}^3 \psi_R) \pt_m + \DD^2 \pt_k \psi_R \\
& = 4 ( \partial_{klm}^3 \psi_R) \partial^2_{lm} + 4 [ \pt_l, (\partial^3_{klm} \psi_R) \pt_m] + \DD^2 \pt_k \psi_R  \\
& =  4 ( \partial_{klm}^3 \psi_R) \partial^2_{lm}  + 4 (  \DD \pt^2_{km} \psi_R) \pt_m + \DD^2 \pt_k \psi_R.
\end{align*}
Next, we use $-\DD (-\DD)^{-1} = \mathds{1}$ and combine the identities above to conclude that
\be
\begin{aligned}
Z & = [\DD^2, \pt_k \psi_R] (-\DD)^{-1} \pt_k \psi_R \\
\phantom{Z} & = -4  \pt_l (\pt^2_{kl} \psi_R) \pt_k \psi_R + 2 (\DD \pt_k \psi_R) \pt_k \psi_R + 4 (\pt^3_{klm} \psi_R) \pt_{lm}^2 (-\DD)^{-1} \pt_k \psi_R \\
\phantom{Z} & \quad + 4 ( \DD \pt^2_{km} \psi_R) \pt_m (-\DD)^{-1} \pt_k \psi_R +  (\DD^2 \pt_k \psi_R) (-\DD)^{-1} \pt_k \psi_R.
\end{aligned}
\ee
By plugging this into $\LL^{(1)}_R[u]= \langle u, i(Z-Z^*) u \rangle = -2 \,\mathrm{Im} \, \langle u, Z u \rangle$ and recalling the identity \eqref{eq:mastercutoff}, an integration by parts for the top order term $4 i \pt_l (\pt^2_{kl} \psi_R) (\pt_k \psi_R)$ yields
\begin{align*}
\LL^{(1)}_R[u] = 4 \, \MM_R[u] - \sum_{\nu=1}^4 \RR_\nu[u],
\end{align*}
with the remainder terms

$$
\RR_1[u] := 4\,  \mathrm{Im} \int_{\R^d} (\DD \pt_k \psi_R) (\pt_k \psi_R) |u|^2 ,
~
\RR_2[u] := 8 \, \mathrm{Im} \int_{\R^d} \overline{u} (\pt^3_{klm} \psi_R) \pt^2_{lm} (-\DD)^{-1} (\pt_k \psi_R) u,
$$

$$
\RR_3[u] := 8\, \mathrm{Im} \int_{\R^d} \overline{u} \, (\DD \pt^2_{km} \psi_R) \pt_m (-\DD)^{-1} (\pt_k \psi_R) u,
$$
$$
\RR_4[u]:= 2 \, \mathrm{Im} \, \int_{\R^d} \overline{u} \, (\DD^2 \pt_k \psi_R) (-\DD)^{-1} (\pt_k \psi_R) u.
$$

As a next step, we claim that
\be \label{ineq:RR}
\left |\RR_\nu[u] \right | \lesssim \| u \|_{L^2}^2 = \cO ( 1 ).
\ee
for $\nu=1, \ldots, 4$. Indeed, we first note that
$$
\RR_1[u] = 0,
$$
since the integrand is real-valued. To estimate $\RR_2[u]$, we use $\| \partial_{lm}^2 (-\DD)^{-1} f \|_{L^2} \lesssim \| f \|_{L^2}$ and the Cauchy-Schwarz inequality to get
$$
| \RR_2[u] | \lesssim \| \pt^3_{klm} \psi_R \|_{L^\infty} \| \pt_k \psi_R \|_{L^\infty}  \|u \|_{L^2}^2 \lesssim R^{-1} \cdot R \|u\|_{L^2}^2 \lesssim \|u\|_{L^2}^2.
$$
Next, we use that $\| \pt_m (-\DD)^{-1} f \|_{L^2} \lesssim \| (-\DD)^{-\frac 1 2} f \|_{L^2}$ together with the Cauchy-Schwarz and the weak Young inequalities. This gives us
\begin{align*}
| \RR_3[u] | & \lesssim \| (\DD \pt^2_{km} \psi_R) u \|_{L^2} \| (-\DD)^{-\frac 1 2} ( \pt_k \psi_R u ) \|_{L^2}  \lesssim \| \DD \pt^2_{km} \psi_R \|_{L^\infty} \| u \|_{L^2} \| \pt_k \psi_R u \|_{L^{\frac{2d}{d+2}}} \\
& \lesssim  \| \DD \pt^2_{km} \psi_R \|_{L^\infty} \| \pt_k \psi_R \|_{L^d} \| u \|_{L^2}^2 \lesssim R^{-2} \cdot R^2 \| u \|_{L^2}^2 \lesssim \| u \|_{L^2}^2,
\end{align*}
using that $\| \pt_k \psi_R \|_{L^d} \lesssim R \cdot \left | \{ |x| \leq 10R \} \right |^{\frac 1 d} \lesssim R^2$ thanks to \eqref{eq:psi1}. Finally, we note $(-\DD)^{-1} = (-\DD)^{-\frac 1 2} (-\DD)^{- \frac 1 2}$ and apply the weak Young inequality once again to find that
\begin{align*}
| \RR_4[u] | & \lesssim \| (-\DD)^{-\frac 1 2} (\DD^2 \pt_k \psi_R) u \|_{L^2} \| (-\DD)^{-\frac 1 2} (\pt_k \psi_R u) \|_{L^2} \\
& \lesssim \| (\DD^2 \pt_k \psi_R) u \|_{L^{\frac{2d}{d+2}}} \| \pt_k \psi_R u \|_{L^{\frac{2d}{d+2}}} \lesssim \| \DD^2 \pt_k \psi_R \|_{L^d} \| \pt_k \psi_R \|_{L^d} \| u \|_{L^2}^2 \\
& \lesssim R^{-2} \cdot R^2 \| u \|_{L^2}^2 \lesssim \|u \|_{L^2}^2, 
\end{align*}
since we have $\| \DD^2 \pt_k \psi_R \|_{L^d} \lesssim R^{-3} \cdot \left | \{ |x| \leq 10R \} \right |^{\frac 1 d} \lesssim R^{-2}$ by \eqref{eq:psi1} and $\| \pt_k \psi_R \|_{L^d} \lesssim R^2$ as shown above. This completes the proof of estimate \eqref{ineq:RR}.

\medskip
{\bf Step 3 (Analysis of $\LL_R^{(2)}$).} Let us now turn to term $\LL_R^{(2)}$ arising from the commutator of $\Psi_R$ with the lower-order dispersion. By using that $[\DD, (-\DD)^{-1}] = 0$ and $[A,BC] = [A,B]C + B[A,C]$, we calculate
$$
[\DD, i \Psi_R] = i[\DD, \pt_k \psi_R] (-\DD)^{-1} \pt_k \psi_R + i \pt_k \psi_R (-\DD)^{-1} [\DD, \pt_k \psi_R] =: i ( \tilde{Z} - \tilde{Z}^*)
$$
with $\tilde{Z} :=  [\DD, \pt_k \psi_R] (-\DD)^{-1} \pt_k \psi_R$. We proceed by noticing that 
\begin{align*}
[\DD, \pt_k \psi_R] & = [\pt_l, \pt_k \psi_R] \pt_l + \pt_l [\pt_l, \pt_k \psi_R] \\
& = 2 [\pt_l, \pt_k \psi_R] \pt_l + [ \pt_l, [\pt_l, \pt_k \psi_R]] = 2 ( \pt^{2}_{kl} \psi_R) \pt_l + \DD \pt_k \psi_R.
\end{align*}
Since $\LL_R^{(2)}[u] = -\mu \langle u, i (\tilde{Z} - \tilde{Z}^*) u \rangle = 2 \mu \, \mathrm{Im} \, \langle u, \tilde{Z} u \rangle$, we obtain
$$
\LL_R^{(2)}[u] = \tilde{\RR}_1[u] + \tilde{\RR}_2[u]
$$
with
$$
\tilde{\RR}_1[u] = 4\mu \, \mathrm{Im} \int_{\R^d} \overline{u} (\pt^2_{kl} \psi_R) \pt_l (-\DD)^{-1} \pt_k \psi_R u ,
$$
$$
\tilde{\RR}_2[u] = 2 \mu \, \mathrm{Im} \int_{\R^d} \overline{u} (\DD \pt_k \psi_R) (-\DD)^{-1} \pt_k \psi_R u.
$$
Next, we claim that
\be \label{ineq:RR2}
|\tilde{\RR}_\nu[u] | \lesssim  |\mu| R^2 \| u \|_{L^2}^2 = \cO( |\mu| R^2 ) 
\ee
for $\nu=1,2$. To see this, we use that $\| \pt_l (-\DD)^{-1} f \|_{L^2} \lesssim \| (-\DD)^{-\frac 1 2} f \|_{L^2}$ and apply the Cauchy-Schwarz and weak Young inequalites to deduce
\begin{align*}
|\tilde{\RR}_1[u]| & \lesssim |\mu| \|\pt^2_{kl} \psi_R\, u \|_{L^2} \| (-\DD)^{-\frac 1 2} ( \pt_k \psi_R\, u) \|_{L^2} \lesssim |\mu| \| \pt^2_{kl} \psi_R \|_{L^\infty} \|u\|_{L^2} \| \pt_k \psi_R\, u \|_{L^{\frac{2d}{d+2}}} \\
& \lesssim |\mu| \| \pt^2_{kl} \psi_R \|_{L^\infty} \| \pt_k  \psi_R\|_{L^d} \| u \|_{L^2}^2 \lesssim |\mu| R^2 \| u \|_{L^2}^2, 
\end{align*}
since $\| \pt^2_{kl} \psi_R \|_{L^\infty} \lesssim 1$ and $\| \pt_k \psi_R \|_{L^d} \lesssim R^2$. Next, by writing $(-\DD)^{-1} = (-\DD)^{-\frac 1 2} (-\DD)^{- \frac 1 2}$ again, another application of the weak Young inequality likewise yields that
\begin{align*}
|\tilde{\RR}_2[u] | & \lesssim |\mu| \| (-\DD)^{- \frac 1 2} ( (\DD \pt_k \psi_R) u) \|_{L^2} \| (-\DD)^{-\frac 1 2} (\pt_k \psi_R u) \|_{L^2} \\
& \lesssim |\mu| \| (\DD \pt_k \psi_R) u \|_{L^{\frac{2d}{d+2}}} \| \pt_k \psi_R u \|_{L^{\frac{2d}{d+2}}}  \lesssim |\mu| \| \DD \pt_k \psi_R \|_{L^d} \| \pt_k \psi_R \|_{L^d} \| u \|_{L^2}^2 \\
& \lesssim |\mu| R^2 \| u \|_{L^2}^2,
\end{align*}
using the bounds $\| \DD \pt_k \psi_R \|_{L^d} \lesssim 1$ and $\| \pt_k \psi_R \|_{L^d} \lesssim R^2$.  This shows that \eqref{ineq:RR2} holds.

The proof of Lemma \ref{lem:VR} is now complete. 
\end{proof}

Next, we prove the following bounds for the nonlinear commutator term $\NN_R[u]$ in the class of radial functions.

\begin{lemma}[Bounds for $\NN_R$] \label{lem:NR}
Let $d \geq 3$. Suppose $\frac{4}{d} \leq \sigma < \sigma_*$ and define $\delta = d \sigma -4 \geq 0$. For any radial function $u \in H^2(\R^d)$, it holds that
$$
|\NN_R[u] | \lesssim C_\eps (\| u \|_{L^2}) R^a \| \DD u \|_{L^2}^{\frac{1}{2} ( \delta + a )}  ,
$$
where
$$
a = \begin{dcases*}  \eps + a_0(d, \sigma) & for $d \geq 6$,\\
\frac{6-d}{2} & for $3 \leq d \leq 5$, \end{dcases*}
$$
with any  $0 < \eps < 2$ and
$$
a_0(d, \sigma) = \begin{dcases*} 0 & for $d \geq 8$,\\
0 & for $d=7$ and $\sigma \in [\frac{7}{12}, \sigma_*)$,\\
0 & for $d=6$ and $\sigma \in [\frac{7}{10}, \sigma*)$, \end{dcases*} 
$$
and 
$$ 
 a_0(d, \sigma) =  \begin{dcases*}  \frac{7}{12} (7-12 \sigma) & for $d =7$ and $\sigma \in [\frac{4}{7}, \frac{7}{12})$, \\
\frac{3}{5} (7-10 \sigma) & for $d = 6$ and $\sigma \in [\frac{2}{3}, \frac{7}{10})$. 
 \end{dcases*}
$$
\end{lemma}

\begin{remarks*}
{\em  1.~By scaling arguments, it is easy to see that for the estimate $|\NN_R[u]| \lesssim C(\| u \|_{L^2}) R^{a} \| \DD u(t) \|_{L^2}^b$ to hold,
the exponents $a$ and $b$ have to satisfy the relation
\be \label{eq:ab}
-a + 2b  = \delta.
\ee
In particular, if we assume that $a \geq 0$, we get the lower bound $2b \geq  \delta$ with $\delta=d \sigma -4$. As we will see in the proof of Theorem \ref{thm:rate} below, the condition $b < 1$ naturally enters, which leads to the upper bound $\delta < 2$ meaning that $\sigma < \frac 6 d$ holds. Note that the bounds of Lemma \ref{lem:NR} will in fact impose the condition $\sigma < \min\{ \frac 3 d + \frac 1 2, \frac 6 d \}$ in order that $b < 1$ holds. Note that this extra condition on $\sigma$ becomes redundant for $d \geq 12$, since $\sigma < \sigma_*= \frac{4}{d-4}$. 

2.~The proof of Lemma \ref{lem:NR} given below will make use of Newton's theorem (in particular, we will make essential use of this fact for $d \geq 7$.). Alternatively, one could avoid making use of this special identity for $(-\DD)^{-1}$ at all and only work with the weak Young, Strauss and Gagliardo-Nirenberg inequalities at the expense of obtaining weaker bounds for $\NN_R[u]$. 
}
\end{remarks*}

\begin{proof}
First, we note that
$$
\NN_R[u] = - 2 \, \mathrm{Im} \int_{\R^d} \overline{u} \pt_k \psi_R (-\DD)^{-1} \pt_k \psi_R |u|^{2 \sigma} u .
$$
We discuss the cases $3 \leq d \leq 5$, $d = 6$, and $d \geq 7$ separately as follows.

\medskip
{\bf Case 1: $3 \leq d \leq 5$.} First, we recall the pointwise bound
\be \label{ineq:newton1}
\left | \left (  (-\DD)^{-1} f \right )  (x)  \right | \lesssim \frac{1}{|x|^{d-2}} \left ( \int_{\R^d} | f(y) | \, dy \right ) \quad \mbox{when $x \neq 0$},
\ee
for any radial function $f \in L^1(\R^d)$ and $d \geq 3$. This bound can be deduced, e.\,g., from Newton's theorem, see \cite[Theorem 9.7]{LiLo-01} and the proof given there. Alternatively, a more stable argument that yields the pointwise bound \eqref{ineq:newton1}, and can be generalized to Riesz potentials $(-\DD)^{-\alpha}$, $1/2 < \alpha <  d/2$, can be inferred from \cite[Corollary 2.3]{Du-13}. 

Applying \eqref{ineq:newton1} to the radial function $f = |\pt_k \psi_R| |u|^{2 \sigma +1} \in L^1(\R^d)$ and using that $| \pt_k \psi_R | \lesssim |x|$ and $\mathrm{supp} \, ( \pt_k \psi_R) \subset \{ |x| \lesssim R \}$  by \eqref{eq:psi1}, we deduce
$$
\left | \NN_R[u] \right | \lesssim \left (  \int_{|x| \lesssim R} \frac{|u(x)|}{|x|^{d-3}} \, dx \right ) \cdot  \left ( \int_{|y| \lesssim R} |y| |u(y)|^{2 \sigma +1} \, dy \right ) =: A \cdot B.
$$
Next, we note that
$$
A \lesssim  \| u \|_{L^2} \left \| \frac{1}{|x|^{d-3}} \right  \|_{L^2(|x| \lesssim  R)} \lesssim \| u \|_{L^2} R^{\frac{6-d}{2}},
$$
and
$$
B \lesssim \int_{\R^d}�\left ( |y|^{\frac{d-1}{2}} |u(y)| \right )^{2 \gamma} |u(y)|^{2 \sigma + 1 -2 \gamma} \, dy  \lesssim \| \Delta u \|_{L^2}^{\frac{\gamma}{2}} \| u \|_{L^\beta}^{\beta} ,
$$
where we have used the Strauss inequality and introduced the exponents
$$
\gamma :=\frac{1}{d-1} \quad \mbox{and} \quad \beta := 2 \sigma + 1 -2\gamma.
$$ 
Notice that $\beta \in [2, 2 \sigma_* + 2)$ for $\sigma \in [ \frac 4 d ,\sigma_*)$ and $3 \leq d \leq 5$. Thus we can apply the Gagliardo-Nirenberg inequality to bound $\| u \|_{L^\beta}$ in $B$, whence it follows that
$$
\left | \NN_R[u] \right |  \lesssim C(\| u \|_{L^2}) R^{\frac{6-d}{2}} \, \| \DD u \|_{L^2}^{\frac{\gamma}{2}} \| \DD u \|_{L^2}^{\frac{d}{2} (\sigma-\gamma) - \frac{d}{4}}  = C(\| u \|_{L^2}) R^{a} \| \DD u \|_{L^2}^{\frac{1}{2}(\delta +a)}  
$$
with $a = \frac{6-d}{2}$, which is the bound asserted in Lemma \ref{lem:NR} when $3 \leq d \leq 5$.

\medskip
{\bf Case 2: $d=6$.}  Let $A$ and $B$ be as above.  Since the function $|x|^{-d+3}$ does not belong to $L^2_{\mathrm{loc}}(\R^d)$ anymore for $d \geq 6$,  we modify the previous argument to control $A$ as follows. Let $\eps \in (0,2)$ and we estimate 
\be \label{ineq:A6}
A  \leq C_\eps \left \| \frac{1}{|x|^{3-\eps}} \right  \|_{L^2(|x| \lesssim R)}  \| |\nabla|^{\eps} u \|_{L^2}  \lesssim C_\eps(\| u \|_{L^2}) R^{\eps} \| \DD u \|_{L^2}^{\frac{\eps}{2}} 
\ee 
where we used the Hardy-type inequality $\| |x|^{-\eps} u \|_{L^2} \leq C_\eps \| |\nabla|^{\eps} u \|_{L^2}$ and $\| |\nabla|^\eps u \|_{L^2} \leq \| u \|_{L^2}^{1 - \frac{\eps}{2}} \| \DD u \|_{L^2}^{\frac{\eps}{2}}$ for $\eps \in (0,2)$. 

Next, we let $\gamma=\frac{1}{d-1}=\frac{1}{5}$ and $\beta = 2 \sigma + 1- 2\gamma$ as above. However, we notice that 
$$
\beta = 2 \sigma+1-2 \gamma < 2 \quad \mbox{if $\frac{2}{3} \leq \sigma < \frac{7}{10}$},
$$
Thus we cannot make use of the Gagliardo-Nirenberg inequality \eqref{ineq:GNS} to control $B$ when $\sigma \in [\frac 2 3, \frac{7}{10})$. Instead, by making use of the Strauss inequality and H\"older's inequality, we obtain
\begin{align*}
B & = \int_{|y| \lesssim R} \left ( |y|^\frac{d-1}{2} |u| \right )^{2 \gamma} |u(y)|^{\beta} \, dy  \lesssim C( \| u \|_{L^2})  \left ( \int_{|y| \lesssim R} 1 \, dy \right )^{\frac{7-10\sigma}{10}}   \| \DD u \|_{L^2}^{\frac{1}{10}}  \\
& \lesssim C(\| u \|_{L^2})R^{\frac{3(7-10 \sigma)}{5}} \| \DD u \|_{L^2}^{\frac{1}{10}},
\end{align*}
provided that $\sigma \in [\frac{2}{3}, \frac{7}{10})$. In view of \eqref{ineq:A6}, we deduce the claimed bound for $\NN_R[u]$ when $d=6$ and $\sigma \in [\frac 2 3, \frac{7}{10})$. 

Let us now consider $\sigma \in [\frac{7}{10}, \sigma_*)$. In this case, we have $\beta \in [2, 2 \sigma_\star + 2)$ and hence we can use the GN-inequality to bound $\| u \|_{L^\beta}^\beta$  and we obtain
$$
\left | \NN_R[u] \right | \lesssim A \cdot B \lesssim C(\| u \|_{L^2}, \eps) R^\eps \| \DD u \|_{L^2}^{\frac{1}{2}(\delta + \eps)}. 
$$
This completes the proof of Lemma \ref{lem:NR} for $d=6$.

\medskip
{\bf Case 3: $d \geq 7$.} We shall now fully exploit Newton's theorem, which states that
\be \label{eq:newtonfull}
((-\DD)^{-1} f)(x) = \frac{1}{|x|^{d-2}} \int_{|y| \leq |x|} f(y) \, dy + \int_{|x| < |y|} \frac{f(y)}{|y|^{d-2}} \, dy \quad \mbox{for $x \neq 0$},
\ee
for any radial $f \in L^1(\R^d, \langle x \rangle^{2-d} dx)$ and $d \geq 3$; see, e.\,g.,\cite[Theorem 9.7]{LiLo-01} and the proof given there. By making use of \eqref{eq:newtonfull} with the radial function $f = |\pt_k \psi_R||u|^{2 \sigma+1}$, and the estimate $| \pt_k \psi_R | \lesssim |x|$ we deduce that
\be
\begin{aligned} \label{ineq:NNnewton}
& \left | \NN_R[u] \right | \lesssim \int_{|x| \lesssim R} |u(x)| |x| \left ( \frac{1}{|x|^{d-2}} \int_{|y| \leq |x| \lesssim R} |y| |u(y)|^{2 \sigma +1} \, dy \right . \\
& \phantom{\left | \NN_R[u] \right | \lesssim \int_{|x| \lesssim R} |u(x)| |x| \left ( \frac{1}{|x|^{d-2}} \right .}  \left . + \int_{|x| < |y| \lesssim R} \frac{1}{|y|^{d-3}}  |u(y)|^{2 \sigma +1} \, dy \right ) dx .
\end{aligned}
\ee
Let us first assume that $d \geq 8$ holds. In this case, we observe that
$$
 \left \| y \right \|_{L^p(|y| \leq |x| \lesssim R)} \leq \| y \|_{L^p(|y| \leq |x|)} \lesssim |x|^{1+ \frac{d}{p}} \quad \mbox{for $p \in [1,\infty]$}, 
$$
$$
 \left \| \frac{1}{|y|^{d-3}} \right \|_{L^p(|x|< |y| \lesssim R)} \leq \left \| \frac{1}{|y|^{d-3}} \right \|_{L^p(|x|< |y|)} \lesssim |x|^{-(d-3)+\frac{d}{p}} \quad \mbox{for $p \in (\frac{d}{d-3}, \infty]$}. 
$$ 
Using these bounds with with $p=\infty$ for $d=8$ and with $p= \frac{2d}{d-8}$ for $d>8$, we apply H\"older's inequality to \eqref{ineq:NNnewton} and find that
\begin{align*}
\left | \NN_R[u] \right | & \lesssim \left ( \int_{|x| \lesssim R} \frac{|u(x)|}{|x|^{d/2}} \, dx \right ) \cdot \left \| u \right \|_{L^{\frac{2d}{d+8} \cdot (2 \sigma+1)}}^{2\sigma+1} \lesssim C_\eps(\| u \|_{L^2}) R^\eps \| |\nabla|^\eps u \|_{L^2} \| \DD u \|^{\frac{\delta}{2}} \\
& \lesssim C_\eps(\| u \|_{L^2}) R^\eps \| \DD u \|_{L^2}^{\frac{1}{2}(\delta+\eps)},
\end{align*}
for any $0 < \eps < 2$. Note we used the Hardy-type inequality $\| |x|^{-\eps} u \|_{L^2} \lesssim C_\eps \| |\nabla|^{\eps} u \|_{L^2} \lesssim C_\eps(\| u \|_{L^2}) \| \DD u \|_{L^2}^\epsilon$ to estimate the compactly supported integral above. Notice also that in the second inequality above we used the GN-inequality, which is applicable here due to the fact that $\frac{2d}{d+8}(2 \sigma+1) \in [2, 2 \sigma_* +2)$ holds for $\sigma \in [\frac{4}{d}, \sigma_*)$, as one easily checks.

It remains to discuss the case $d=7$. Here we have to modify the previous arguments with the use of the Strauss inequality as follows: Going back to \eqref{ineq:NNnewton} and splitting $|y| = |y|^{\frac 1 2} |y|^{\frac 1 2}$, we find that
\be \label{ineq:d7}
\begin{aligned}
& \left | \NN_R[u] \right |  \lesssim \left ( \int_{|x| \lesssim R} \frac{|u(x)|}{|x|^{7/2}} \, dx \right ) \cdot \left ( \int_{|y| \lesssim R} |y|^{\frac 1 2} |u(y)|^{2 \sigma + 1} \, dy \right ) \\
& \phantom{\left | \NN_R[u] \right |}  \lesssim C_\eps(\| u \|_{L^2}) R^\eps \| \DD u \|_{L^2}^{\frac{\eps}{2}} \| \DD u \|_{L^2}^{\frac{1}{24}} \left ( \int_{|y| \lesssim R} |u(y)|^{2 \sigma + 1 - \frac{1}{6}}  \, dy \right ) .
\end{aligned}
\ee
Now, we note that $2\sigma +1 - \frac{1}{6} < 2$ for $\sigma \in [\frac{4}{7}, \frac{7}{12})$. Therefore, in this range of $\sigma$, we use H\"older's inequality and the compact support to get the bound
\be \label{ineq:d7_2}
 \int_{|y| \lesssim R} |u(y)|^{2 \sigma + 1 - \frac{1}{6}}  \, dy  \leq \left ( \int_{|y| \lesssim R} 1 \, dy \right)^{\frac{7-12 \sigma}{12}} \| u \|_{L^2} \lesssim R^{\frac{7(7-12\sigma)}{12}} \| u \|_{L^2},
\ee
provided that $\sigma \in [\frac{4}{7}, \frac{7}{12})$. Furthermore, from the GN-inequality \eqref{ineq:GNS} we obtain
\be \label{ineq:d7_3}
\int_{|y| \lesssim R}  |u(y)|^{2 \sigma + 1 - \frac{1}{6}} \, dy \lesssim C(\| u \|_{L^2}) \| \DD u \|_{L^2}^{\frac{d(12 \sigma-7)}{24}}
\ee 
when $\sigma \in [\frac{7}{12}, \sigma_*)$. If we plug the bounds \eqref{ineq:d7_2} and \eqref{ineq:d7_3} into \eqref{ineq:d7}, we obtain the claimed bounds for $\NN_R[u]$ for $d = 7$. 

The proof of Lemma \ref{lem:NR} is now complete.
\end{proof}

We conclude this section by showing a space-time bound for $\NN_R[u]$ for mass-supercritical exponents $\sigma > \frac{4}{d}$, which will be essential in the proof of Theorem \ref{thm:rate} in the next section.

\begin{lemma}[Space-Time Bounds for $\mathcal{N}_R$] \label{lem:NR2}
Let $d \geq 3$, $\frac 4 d < \sigma < \sigma_*$, and define $\delta = d\sigma-4>0$. Suppose $u \in C^0([0,T); H^2(\R^d))$ is radial. Let $a>0$ be as in Lemma \ref{lem:NR} and assume that
$$
b:= \frac{1}{2} (\delta + a) < 1.
$$
Furthermore, we define
$$
I(t_1, t_0):= \int_{t_0}^{t_1} (t_1 -t) \| \DD u (t) \|_{L^2}^2 \, dt.
$$
for $[t_0,t_1] \subset [0, T)$. Then we have
$$
\int_{t_1}^{t_0} | \NN_R[u(t)] | \, dt \leq C(u_0) \left ( \frac{(t_1-t_0)^2}{\eta^2 R^{4/\alpha}} +   R^{2  a- \frac{4}{\alpha} (1-b)} + \eta I(t_1,t_0) \right )
$$
for any $\eta > 0$ and the exponent
$$
\alpha = \frac{\sigma(d-1)}{4- \sigma}.
$$
\end{lemma}

\begin{remark*}
{\em The role of the exponent $0 < \alpha < 1$ will become clear in the proof of Theorem \ref{thm:rate} below. }
\end{remark*}
\begin{proof}
From Lemma \ref{lem:NR} we recall that
$$
| \NN_R[u(t)] | \leq C(u_0) R^a \| \DD u(t) \|_{L^2}^b
$$
with some constants $a \geq 1$ and $0 < b < 1$. Integrating this bound on $[t_0, t_1]$ and using H\"older's inequality, we find
\begin{align*}
\int_{t_0}^{t_1} |\NN_R[u(t)] | \, dt & \leq C(u_0) R^a \int_{t_0}^{t_1} \| \DD u(t) \|_{L^2}^b \, d t  \\
& \leq C(u_0) R^a \left ( \int_{t_0}^{t_1} (t_1-t)^{-\frac{b}{2-b}} \, dt \right )^{\frac{2-b}{2}} \left ( \int_{t_0}^{t_1} (t_1 -t )\| \DD u(t) \|_{L^2}^2 \, dt \right )^{\frac{b}{2}} \\
& = C(u_0) R^a (t_1-t_0)^{1-b} I(t_0,t_1)^{\frac b 2} .
\end{align*}
We let $\eta > 0$ and invoke Young's inequality twice to deduce that
\begin{align*}
\int_{t_0}^{t_1} |\NN_R[u(t)] | \, dt & \leq C(u_0) \left ( \eta^{-1} R^{\frac{2a}{2-b}} (t_1 -t_0)^{\frac{2 (1-b)}{2-b}} + \eta I(t_0, t_1) \right ) \\
& \leq C(u_0) \left ( \frac{(t_1-t_0)^2}{\eta^2 R^{4/\alpha}} + R^{2  a- \frac{4}{\alpha} (1-b)}+ \eta I(t_1, t_0) \right ),
\end{align*}
where we used that
\be \label{ineq:foreverYoung}
\begin{aligned}
R^{\frac{2 a}{2 - b}} (t_1-t_0)^{\frac{2 (1-b)}{2-b}} & = R^{\frac{1}{2-b}\,( 2 a + \frac{4}{\alpha} (1-b) )} \left (\frac{(t_1 - t_0)^2}{R^{\frac{4}{\alpha}}}\right )^{\frac{1-b}{2-b}} \\
& \lesssim \frac{(t_1 - t_0)^2}{\eta R^{\frac{4}{\alpha}}} + \eta R^{2 a + \frac{4}{\alpha} (1-b)}.
\end{aligned}
\ee
The proof of Lemma \ref{lem:NR2} is now complete.
\end{proof}

\section{Universal Upper Bound on Blowup Rate} \label{sec:blowup_rate}

This section is devoted to the proof of Theorem \ref{thm:rate}. Inspired by the work on classical NLS by Merle-Rapha\"el-Szeftel in \cite{MeRaSz-14}, we will make essential use of the localized Riesz bivariance estimates, derived in Section \ref{sec:locrieszvar} above. 

\subsection{Proof of Theorem \ref{thm:rate}}
We assume that $d \geq 3$, $\mu \in \R$, and $0 < s_c < 2$ with the additional condition that $$\frac{4}{d} < \sigma < \min \left \{ \frac{3}{d} + \frac{1}{2}, \frac{6}{d} \right \}.$$
Let $u_0 \in H^2(\R^d)$ be radial and suppose the corresponding solution $u \in C^2([0,T); H^2(\R^d))$ of \eqref{eq:bNLS} blows up at some finite time $0 < T < +\infty$. Furthermore, we let 
\be \label{cond:R}
0 < R \leq \min \left \{ 1, |\mu|^{-2} \right \}
\ee
be a constant that will be chosen sufficiently small  depending on $u_0$, $d$, and $\sigma$. 

\begin{remark*}{\em
In the proofs of Theorems \ref{thm:blowup} and \ref{thm:energy} above, we took $R \gg 1$ to be {\em sufficiently large} to ensure that certain error terms could be neglected. In contrast to this, we emphasize that we will have to choose $R \ll 1$ to be {\em sufficiently small} below.}
\end{remark*}

Following the notation  in Section \ref{sec:locrieszvar} above, we use 
$$
\MM_R[u(t)] := \MM_{\phi_R}[u(t)] \quad \mbox{and} \quad \VV_R[u(t)] := \VV_{\psi_R}[u(t)]
$$
to denote the localized virial and Riesz-bivariance defined where $\phi_R$ and $\psi_R$ were defined in \eqref{eq:phi_Riesz} and \eqref{def:psiR} respectively. Finally, we suppose that
$$
0 < t_0 < t_1 < T 
$$
are two times that will be chosen below sufficiently close to $T$ depending only on $u_0$, $d$, and $\sigma$. Without loss of generality, we assume that $|t_0 -t_1| \leq 1$ holds. For the rest of the proof, we let $C(u_0) > 0$ denote a constant that only depends on $u_0$, $d$, and $\sigma$. 

The proof of Theorem \ref{thm:rate} will now be arranged into two steps as follows.

\medskip
{\bf Step 1 (Integral Bounds).} We start by bounding the error term (including those in $\cO(|\mu|)$) in Lemma \ref{lem:MR} as follows
\be
\begin{aligned}
& \cO  \left ( R^{-4} + \left ( R^{-2} + |\mu| \right ) \| \nabla u(t) \|_{L^2}^2 + R^{-\sigma(d-1)} \| \nabla u(t) \|_{L^2}^\sigma + |\mu| R^{-2} \right ) \\
& \leq C(u_0)  \left ( R^{-4} +  R^{-2} \| \nabla u(t) \|_{L^2}^2 + R^{-\sigma (d-1)} \| \nabla u(t) \|_{L^2}^\sigma  \right ) \\
& \leq C(u_0) \left ( \eta^{-1} R^{-4} + \eta^{-1} R^{-4/\alpha}   + \eta \| \DD u(t) \|_{L^2}^2 \right ) \\
& \leq C(u_0) \left ( \eta^{-1} R^{-4/\alpha} + \eta \| \DD u(t) \|_{L^2}^2 \right )
\end{aligned}
\ee
where we used that $\| \nabla u \|_{L^2} \leq \| u \|_{L^2}^{\frac 1 2} \| \DD u \|_{L^2}^{\frac 1 2}$ together with Young's inequality to insert some small number $\eta > 0$ to be chosen later. Moreover, we set the following interpolation exponent
\be
\alpha = \frac{4 - \sigma}{\sigma(d-1)}. 
\ee
Note that in the last step above, we used that $0 < \alpha < 1$ thanks to the fact that $4/d < \sigma < 4$ by assumption. 

Thus, by choosing $0 < \eta < \delta/2$ sufficiently small (recall we set $\delta  := d\sigma - 4$), the differential inequality in Lemma \ref{lem:MR} yields that
\begin{align*}
\frac{d}{dt} \MM_R[u(t)] & \leq 4 d \sigma E[u_0] - \left(2\delta - \eta\right) \| \DD u(t) \|_{L^2}^2 + \frac{C(u_0)}{\eta R^{4 / \alpha}} \\
& \leq - \delta \| \DD u(t) \|_{L^2}^2 + \frac{C(u_0)}{\eta R^{4 / \alpha}} \quad \mbox{for $t \in [t_0, T)$},
\end{align*}
provided that $t_0 < T$ is sufficiently close to $T$ and using that $\| \DD u(t) \|_{L^2} \to +\infty$ as $t \to T$. Integrating this bound on an arbitrary time interval $[t_0, t] \subset [t_0, t_1]$ leads to
\begin{align*}
 \MM_R[u(t)] & \leq  - \delta \int_{t_0}^t \| \DD u(\tau) \|_{L^2}^2 \, d \tau+  \MM_R[u(t_0)]  + \frac{C(u_0)}{\eta R^{4/\alpha}} (t_1-t_0) \\
 & \leq  - \delta \int_{t_0}^t \| \DD u(\tau) \|_{L^2}^2 \, d \tau + \frac 1 4 \frac{d}{dt} \VV_{R}[u(t_0)]  \\
 & \quad + C(u_0) \left ( \frac{(t_1-t_0)}{\eta R^{4/\alpha}} + 1 + R^{a} \| \DD u(t_0) \|_{L^2}^b \right ),
\end{align*}
with some $a >0$ and $0 < b < 1$, where we made use of Lemma \ref{lem:VR} and \ref{lem:NR}. If we use the identity in Lemma \ref{lem:VR} once again and integrate the previous inequality on $[t_0, t_1]$, we obtain
\be \label{ineq:VRR} 
\begin{aligned}
&  \VV_\RR[u(t_1)] + 4 \delta \int_{t_0}^{t_1} \int_{t_0}^t \| \DD u(\tau) \|_{L^2}^2 \, d \tau \, dt - \int_{t_0}^{t_1} \NN_R[u(t)] \, d t \\
& \leq \VV_{R}[u(t_0)] + C(u_0) \left ( \frac{(t_1-t_0)^2}{\eta R^{4/\alpha}} +  \left (1 + R^a \| \DD u(t_0) \|_{L^2}^b \right  ) (t_1-t_0)\right ) 
\end{aligned}
\ee
Note that integration by parts on $F(t) = \int_{t_0}^t \| \DD u(\tau) \|_{L^2}^2 \, d \tau$ yields that
$$
\int_{t_0}^{t_1} \int_{t_0}^t \| \DD u(\tau) \|_{L^2}^2 \, d \tau \, dt  =   \int_{t_0}^{t_1} (t_1 - t) \| \DD u(t) \|_{L^2}^2 dt .
$$
Next, we combine the facts that $\VV_R[u(t_1)] \geq 0$ from \eqref{id:VirPos} and $\VV_{R}[u(t_0)] \leq C(u_0) R^4$  from \eqref{ineq:HardyVR} with the previous bound \eqref{ineq:VRR}. Furthermore, we use the time-averaged bound for $\NN_R[u(t)]$ in Lemma \ref{lem:NR2} with $\eta > 0$ sufficiently small to deduce that
\be \label{ineq:blowup1}
\begin{aligned}
& \int_{t_0}^{t_1} (t_1-t) \| \DD u(t) \|_{L^2}^2 \, dt  \\
& \leq C(u_0) \left ( \frac{(t_1-t_0)^2}{\eta^2 R^{4/\alpha}} +   \left (1   + R^a  \| \DD u(t_0) \|_{L^2}^b  \right ) (t_1-t_0) + R^{2  a- \frac{4}{\alpha} (1-b)} + R^4 \right ) . 
\end{aligned}
\ee
Since $0 < b < 1$, we can apply Young's inequality to get
\be \label{ineq:blowup2}
\begin{aligned}
R^a \| \DD u(t_0) \|_{L^2}^b (t_1-t_0) & \lesssim \eta^{-1} R^{\frac{2a}{2-b}} (t_1-t_0)^{\frac{2 (1-b)}{2-b}} + \eta (t_1-t_0)^2 \| \DD u(t_0) \|_{L^2}^2 \\
& \lesssim \frac{(t_1-t_0)^2}{\eta^2 R^{4/\alpha}} + R^{2a + \frac{4}{\alpha}(1-b)} + \eta (t_1 - t_0)^2 \| \DD u(t_0) \|_{L^2}^2,
\end{aligned}
\ee 
where we used \eqref{ineq:foreverYoung} for the last step. Next, we note that
\be \label{ineq:blowup3}
(t_1-t_0) \lesssim \frac{(t_1-t_0)^2}{R^{4/\alpha}} + R^{4/\alpha} \lesssim \frac{(t_1-t_0)^2}{R^{4/\alpha}} + R^4,
\ee
since we have $R^{4 / \alpha} \leq R^4$ due to $0 < \alpha < 1$ and $0 < R \leq 1$. By inserting the bounds \eqref{ineq:blowup2} and \eqref{ineq:blowup3} into \eqref{ineq:blowup1} with $\eta > 0$ sufficiently small, we obtain
\be
\begin{aligned}
& \int_{t_1}^{t_0} (t_1-t) \| \DD u(t) \|_{L^2}^2 \, dt \\
&  \leq C(u_0) \left ( \frac{(t_1-t_0)^2}{\eta^2 R^{4/\alpha}} + \eta (t_1-t_0)^2 \| \DD u(t_0) \|_{L^2}^2 + R^{\rho} + R^4 \right ) .
\end{aligned}
\ee
where we introduce the exponent
\be 
\rho := 2a - \frac{4}{\alpha}(1-b) .
\ee
Now we claim that
\be \label{ineq:rho}
4 > \rho \geq 4 - \left (\frac{3+a}{2}  \right ) \delta \quad \mbox{with $\delta = d \sigma - 4 \in (0,1)$,}
\ee
which, in particular, implies that $R^{4} \leq R^\rho$ for $0 < R \leq 1$. To show \eqref{ineq:rho}, we apply the identities
$$
\frac{1}{\alpha} = \frac{\sigma(d-1)}{4 - \sigma} = \frac{4 + \delta}{4- \delta/(d-1)} \quad \mbox{and} \quad b = \frac{1}{2} (a +\delta),
$$
which lead us to 
\be
\rho  = 2 a + 2 ( 2-a-\delta) \frac{4 + \delta}{4- \delta/(d-1)} .
\ee
As an aside, we remark that this identity shows that 
\be \label{conv:rho}
\rho \to 4 \ \ \mbox{as} \ \ \delta \to 0.
\ee 
Furthermore, we deduce the the lower bound
$$
\rho \geq 2 a + 2 (2-a-\delta) (1 + \delta/4)  = 4 - \frac{\delta}{2} \left ( 2+a + \delta \right ) \delta \geq 4 - \left ( \frac{3+a}{2} \right ) \delta ,
$$
using that $0 < \delta < 1$. 
On the other hand, an elementary calculation shows that
$$
\rho- 4 = - \frac{2 \delta ( (2d-4) + \delta(d-1) + d a)}{4(d-1) - \delta} < 0.
$$
Thus we have shown that \eqref{ineq:rho} holds and we finally obtain
\be \label{ineq:blowupmaster}
\int_{t_0}^{t_1} (t_1 -t) \| \DD u(t) \|_{L^2}^2 \, dt \leq C(u_0) \left ( \frac{(t_1-t_0)^2}{\eta^2 R^{4/\alpha}} + \eta (t_1-t_0)^2 \| \DD u(t_0) \|_{L^2}^2 + R^\rho \right ).
\ee 

\medskip
{\bf Step 2 (Conclusion).} First, we note that right side of \eqref{ineq:blowupmaster} has a finite limit when we take $t_1 \to T < +\infty$.  Furthermore, we  make the optimized ansatz
\be
R = R(t_0) := (T - t_0)^{\frac{2 \alpha}{4 + \rho \alpha}} ,
\ee 
so that $(T-t_0)^2 R^{-4/\alpha} = R^\rho$, and we choose $t_0 < T$ sufficiently close to $T$ in order to guarantee that \eqref{cond:R} holds. With this choice of $R=R(t_0)>0$ and by taking $\eta > 0$ sufficiently small, we deduce
\be \label{ineq:blowupmaster2}
\int_{t_0}^T (T -t) \| \DD u(t) \|_{L^2}^2 \, dt \leq C(u_0) (T-t_0)^{\frac{2 \rho \alpha}{4 + \rho \alpha}} +  (T- t_0)^2 \| \DD u(t_0) \|_{L^2}^2.
\ee
For $t \in [t_0, T)$, we now define the function
\be
g(t):= \int_{t}^T (T-\tau) \| \DD u(\tau) \|_{L^2}^2 \,d \tau .
\ee
Then the integral estimate \eqref{ineq:blowupmaster2} can be written as
$$
g(t)+(T-t) g'(t) \leq C(u_0) (T-t)^{\frac{2 \rho \alpha}{4 + \rho \alpha}}.
$$
Thus we find
$$
\frac{d}{dt} \left ( \frac{g(t)}{T-t} \right ) = \frac{g(t) + (T-t) g'(t) }{(T-t)^2} \leq C(u_0) (T-t)^{\frac{2 \rho \alpha}{4 + \rho \alpha} - 2}.
$$
Hence by integration on $[t_0,t]$ it follows that
$$
\frac{g(t)}{T-t} \leq C(u_0) \left (1 + (T-t)^{\frac{2 \rho \alpha}{4 + \rho \alpha} -1 } \right ) .
$$
Note that $\frac{2 \rho \alpha}{4 + \rho \alpha} < 1$, since $\rho \alpha < 4$ by \eqref{ineq:rho} and $\alpha < 1$. Therefore, we have
\be \label{ineq:g}
g(t) \leq C(u_0) (T-t)^{\frac{2 \rho \alpha}{4 + \rho \alpha}} = C(u_0) (T-t)^{\frac{2 \beta}{1+\beta}} \quad \mbox{with $\beta := \frac{1}{4} \rho \alpha$},
\ee
for $t < T$ sufficiently close to $T$. By choosing $C(u_0) > 0$ larger if necessary, we trivially extend the bound \eqref{ineq:g} to all times $t \in [0,T)$.

Finally, we note that $\beta \to \alpha$ as $\delta \to 0$ (i.\,e.~as $\sigma \to 4/d$) in view of \eqref{conv:rho}. This concludes the proof of Theorem \ref{thm:rate}. \hfill $\blacksquare$



\section{Existence of Blowup for Mass-Critical Case}  \label{sec:blowup_mass}

Let $d \geq 2$, $\mu \geq 0$, and $s_c=0$, i.\,e., we consider the mass-critical exponent
\be
\sigma = \frac{4}{d}.
\ee
We divide the proof of Theorem \ref{thm:mass} into the following steps.

\medskip
{\bf Case 1 (Blowup for $\mu >0$).} In this case, the proof of finite-time blowup for radial data $u_0 \in H^2(\R^d)$ with $E(u_0) < 0$ is similar to the proof of Theorem \ref{thm:blowup} for the mass-supercritical case. In fact, we just exploit the observation that the exponent $\sigma=4/d$ is mass-supercritical with respect to the lower-order NLS dispersion $-\mu \DD$ in \eqref{eq:bNLS}.

Let $\phi_R=\phi(r/R)$ with $R >0$ be a cutoff function as chosen in Section \ref{sec:locvirial} above. Moreover, we use the short-hand notation $\MM_R[u(t)] \equiv \MM_{\phi_R}[u(t)]$ for the localized virial. From Lemma \ref{lem:MR} we obtain that
\be \label{ineq:L2mu}
\begin{aligned}
& \frac{d}{dt} \MM_R[u(t)]  \leq 16 E[u_0] - 4 \mu \| \nabla u(t) \|_{L^2}^2 \\
&\phantom{\frac{d}{dt} \MM_R[u(t)]}  \quad + \cO \left ( R^{-4} + R^{-2} \| \nabla u(t) \|_{L^2}^2 + R^{-4+ \frac{4}{d}} \| \nabla u(t) \|_{L^2}^{\frac{4}{d}} + |\mu| R^{-2} \right ) \\
& \phantom{\frac{d}{dt} \MM_R[u(t)]} \leq 8 E[u_0]  - 2 \mu \| \nabla u(t) \|_{L^2}^2 \quad \mbox{for $t \in [0,T)$},
\end{aligned}
\ee
provided we choose $R> 0$ sufficiently large, where we used that $E[u_0] < 0$ and $\sigma = \frac{4}{d} \leq 2$ by assumption.

Suppose now that $T = +\infty$ holds. Since $E[u_0] < 0$, we see that $\MM_{R}[u(t_1)] \leq 0$ for all $t \geq t_1$ with some sufficiently large time $t_1  >0$. By integrating \eqref{ineq:L2mu}, 
\be
\MM_R[u(t)] \leq -2 \mu \int_{t_1}^t \| \nabla u(s) \|_{L^2}^2 \,ds 
\ee
for all $t \geq t_1$. By the Cauchy-Schwarz inequality, we get $|\MM_R[u]| \lesssim \| \nabla \phi_R \|_{L^2} \| u \|_{L^2} \| \nabla u \|_{L^2}$ and thus we arrive at
\be
\MM_R[u(t)] \leq - A \int_{t_1}^t |\MM_R[u(s)]|^2 \, ds \quad \mbox{with $A:=C(u_0,\mu) > 0$.}
\ee
As in the proof of Theorem \ref{thm:blowup} above, we deduce that $\MM_R[u(t)] \to -\infty$ as $t \to t_*$ for some finite time $t_* < +\infty$. This shows that $u(t)$ cannot exist for all $t \geq 0$. By the blowup alternative, we have finite-time blowup of $u(t)$.

\medskip
{\bf  Case 2 (Blowup for $\mu = 0$).} In this case, the absence of the lower-order dispersion in \eqref{eq:bNLS} requires a more refined analysis of the problem. 

In what follows, we choose the cutoff function $\phi(r)$ that satisfy some additional properties needed, as done in Appendix \ref{sec:cutoff}. Going back to the proof of Lemma \ref{lem:MR} (see, in particular, the proof of Step 2 there), we first observe that
$$
\begin{aligned}
& 4 \langle u, \pt_k (\pt^2_{kl} \DD \phi_R) \pt_l u \rangle + 2 \langle u, \pt_k (\DD^2 \phi_R) \pt_k u \rangle = -4 \int_{\R^d} \pt_r^2 \DD \phi_R \,  |\pt_r u|^2 - 2 \int (\DD^2 \phi_R) |\pt_r u|^2 \\
\end{aligned}
$$
using integration by parts and the formula \eqref{eq:Hessian}. Thus from the calculations in steps 2 and 3 of the proof of Lemma \ref{lem:MR}, and the sign properties \eqref{ineq:phi1} of $\phi_R$ we infer
\be
\begin{aligned} \label{eq:mass_bnls1}
& \frac{d}{dt} \MM_R[u(t)] \leq 16 E(u_0)  - 8 \int_{\R^d}  \left (1 - \pt_r^2 \phi_R \right ) |\pt_r^2 u|^2 \\
& \phantom{\frac{d}{dt} \MM_R[u(t)]} \quad - \int_{\R^d} A_R |\pt_r u|^2 + \int_{\R^d} B_R |u|^{\frac{8}{d} +2 } + \int_{\R^d} (\DD^3 \phi_R) |u|^2
\end{aligned}
\ee
with the radial functions
\be \label{eq:g1}
A_R(r) :=  4\pt_r^2 \DD \phi_R + 2 \DD^2 \phi_R(r), \quad B_R(r) := \frac{8 d}{4+d} \left ( d- \DD \phi_R(r) \right ).  
\ee
Note that $B_R(r) \geq 0$ is nonnegative for all $r \geq 0$ with $B_R(r) \equiv 0$ for $r \leq R$. 

Next, we integrate by parts twice using that $\pt_r^* = - \pt_r - \frac{d-1}{r}$ and obtain
$$
\int_{\R^d} A_R |\pt_r u|^2 = - \int_{\R^d} \overline{u} A_R \pt_r^2 u + \frac 1 2 \int_{\R^d}   \left ( \left ( \pt_r + \frac{d-1}{r} \right )^2 A_R \right ) |u|^2. 
$$
Since $\| \pt_r^j A_R \|_{L^\infty} \lesssim R^{-2-j}$ for $j = 0,1,2$ and $\mathrm{supp} A_R \subset \{ |x| \geq R \}$, we can apply H\"older's and Young's inequality to find that
\be \label{ineq:AR}
\left | \int_{\R^d} A_R |\pt_r u|^2 \right | \lesssim  8\eta R^4 \| A_R \pt_r^2 u \|_{L^2}^2+ \eta^{-1} R^{-4} \| u \|_{L^2}^2 .
\ee
for arbitrary $\eta > 0$. 

Next, we recall that $B_R(r) \equiv 0$ for $r \leq R$ and we invoke the Strauss inequality \eqref{ineq:Strauss}, which yields
\be
\begin{aligned} \label{ineq:BR1}
& \left | \int_{\R^d} B_R |u|^{\frac{8}{d} + 2} \right | \leq  \| u \|_{L^2}^2  \big \| B_R^{\frac{d}{8}} u \big \|_{L^\infty(|x| \geq R)}^{\frac{8}{d}} \\
& \phantom{\left | \int_{\R^d} B_R |u|^{\frac{8}{d} + 2} \right | } \lesssim R^{-4 + \frac{4}{d}} \| u \|_{L^2}^2 \| B_R^{\frac{d}{8}} u \|_{L^2}^{\frac{4}{d}} \| \pt_r ( B_R^{\frac{d}{8}} u ) \|_{L^2(|x| \geq R)}^{\frac{4}{d}} .
\end{aligned}
\ee
Since $|\pt_r( B_R^{\frac{d}{8}} u)|^2 \lesssim | (\pt_r B_R^{\frac{d}{8}}) u |^2 + | B_R^{\frac{d}{8}} \pt_r u|^2$, a similar argument combining integration by parts with Young's inequality, as we used to derive \eqref{ineq:AR} gives us, for any $\eta > 0$, 
\be
\begin{aligned}
& \| \pt_r ( B_R^{\frac{d}{8}} u ) \|_{L^2}^2  \lesssim R^{-2} \| u \|_{L^2}^2 + \int_{\R^d} B_R^{\frac{d}{4}} |u| |\pt_r^2 u| \\
& \phantom{\| \pt_r ( B^{\frac{8}{d}} u ) \|_{L^2}^2} \lesssim \left( \eta^{-\frac{1}{4}} + R^{-2} \right) \| u \|_{L^2}^2 + 8\eta^{\frac{1}{4}} \| B_R^{\frac{d}{4}} \pt_r^2 u \|_{L^2}^2,
\end{aligned}
\ee
where we used the bounds $\| \pt_r^j B^{\frac{d}{4}}_R \|_{L^\infty} \lesssim R^{-j}$ for $j=1,2$ and $\| \partial_r B_R^{\frac{d}{8}} \|_{L^\infty} \lesssim R^{-1}$ (see Appendix \ref{sec:cutoff}) together with  the fact that $B_R(r) \equiv 0 $ for $|x| \leq R$. Going back to \eqref{ineq:BR1}, we readily deduce from Young's inequality for $d=2$ and $R \geq 1$
\be \label{ineq:BR2}
\left | \int_{\R^d} B_R |u|^{\frac{8}{d} +2} \right | \lesssim R^{-2}\left( \eta^{-\frac{1}{4}} + R^{-2} \right) \| u \|_{L^2}^6 + 8\eta^{\frac14}R^{-2} \| u \|_{L^2}^4 \| B_R^{\frac{d}{4}} \pt_r^2 u \|_{L^2}^2
\ee
For $d \geq 3$, a further use of Young's inequality (inserting the small parameter $\eta^{3/4} > 0$) now yields
\be \label{ineq:BR3}
\begin{aligned}
\left | \int_{\R^d} B_R |u|^{\frac{8}{d} + 2} \right | & \lesssim \eta^{-\frac34} R^{-\frac{4(d-1)}{d-2}} \|u\|_{L^2}^{\frac{2(d+2)}{d-2}} + \left( \eta^{\frac12} + \eta^{\frac34}R^{-2} \right) \| u \|_{L^2}^2 + 8\eta \| B_R^{\frac{d}{4}} \pt_r^2 u \|_{L^2}^2 \\
 & \lesssim C( u_0) \left ( \eta^{-1} R^{-2} + \eta^{\frac12} \right ) + 8 \eta \| B_R^{\frac{d}{4}} \pt_r^2 u \|_{L^2}^2.
\end{aligned}
\ee
provided that $R \geq 1$ and $0 < \eta < 1$.  Note that estimate \eqref{ineq:BR3} implies \eqref{ineq:BR2} also for $d=2$ if we choose $R \geq 1$ and $\eta < 1$. Thus by plugging this back into \eqref{eq:mass_bnls1} and recalling that $\|\DD^2 \phi_R \|_{L^\infty} \lesssim R^{-2}$, we obtain
\be
\begin{aligned} \label{ineq:Mmass}
& \frac{d}{dt} \MM_R[u(t)] \leq 16 E[u_0]   - 8 \int_{\R^d} \left (1 - \pt_r^2 \phi_R - \eta \left \{ R^4 (A_R)^2 + (B_R)^{\frac{d}{2}} \right \} \right ) |\pt_r^2 u|^2 \\ 
& \phantom{\frac{d}{dt} \MM_R[u(t)]} \quad + C(u_0) \left ( \eta^{-1} R^{-2} + \eta^{\frac12} \right ), \\
\end{aligned}
\ee
for $R \geq 1$, $0 < \eta < 1$, and $d \geq 2$. 

As a next step, we claim that there is some $\eta_0 > 0$ sufficiently small and independent of $R \geq 1$ such that
\be \label{ineq:trans}
1- \pt_r^2 \phi_R(r) - \eta_0 \left \{ R^4 (A_R(r))^2 + (B_R(r))^{\frac{d}{2}} \right \} \geq  0 \quad \mbox{for $r \geq 0$}.
\ee
The proof of this inequality is worked out in Appendix \ref{sec:cutoff}. 

Since $E[u_0] < 0$ by assumption, we can now choose $0 < \eta \leq \eta_0$ sufficiently small and $R \geq 1$ sufficiently large to conclude from \eqref{ineq:Mmass} that
\be \label{ineq:MR_mu_zero}
\frac{d}{dt} \MM_R[u(t)] \leq 8 E[u_0] \quad \mbox{for all $t \in [0,T)$}.
\ee
Assume now that $T= +\infty$ holds. Then we have $\MM_R[u(t)] \leq 0$ for all $t \geq t_0$ with some sufficiently large time $t_0 \geq 0$. On the other hand, by the Cauchy-Schwarz inequality and integrating, 
\be
- \| \nabla \phi_R \|_{L^\infty} \| u_0 \|_{L^2}^{3/2} \| \DD u(t) \|_{L^2}^{\frac 1 2} \leq \MM_R[u(t)] \leq -8 |E[u_0]| (t-t_0) \quad \mbox{for all $t \geq t_0$.}
\ee
Thus we conclude the following: Either $u(t)$ exists for all times $t \geq 0$ such that
\be \label{ineq:infiniteblowup}
\| \DD u(t) \|_{L^2} \geq C(u_0) (t-t_0)^2 \quad \mbox{for all $t \geq t_0$},
\ee
or the solution $u(t)$ blows up in finite time.

\medskip
{\bf Improved Bounds for $\mu=0$.} We consider the localized Riesz bivariance
$$
\VV_R[u(t)] = \left \langle u(t), \pt_k \psi_R (-\DD)^{-1} \pt_k \psi_R u(t) \right \rangle,
$$
with the cutoff function $\psi_R$ defined in terms of $\phi_R$ via \eqref{def:psiR}, where $\phi_R$ is chosen as above. Choosing $R > 0$ sufficiently large as above, we use Lemma \ref{lem:VR} together estimate \eqref{ineq:MR_mu_zero} and we find that, by integrating in time,
\be \label{ineq:VVmass_critical}
 \VV_R[u(t)] \leq 16 E[u_0] t^2 + \int_0^t \NN_R[u(s)] \, ds + C(u_0) (1+t)  \quad \mbox{for $t \geq 0$}.
\ee  
Moreover, by Lemma \ref{lem:NR}, we have the estimate
\be \label{ineq:NNR_mass}
\left |\NN_R[u] \right | \leq C(u_0, R, b) \| \DD u \|_{L^2}^b,
\ee
where the exponent $b> 0$ is given by
\be \label{eq:bb}
b = \begin{dcases*} \eps & for $d \geq 8$, \\
\frac{1}{24} + \eps & for $d=7$, \end{dcases*}
\qquad b= \begin{dcases*}
\frac{1}{10} + \eps & for $d=6$, \\
\frac{6-d}{4} & for $d = 3,4,5$, \end{dcases*}
\ee 
with arbitrary $0 < \eps < 2$. 

Let $\nu \geq 0$ and suppose there is some constant $C > 0$ such that
\be \label{ineq:infinite}
\| \DD u(t) \|_{L^2} \leq C (1 +  t)^{\nu} \quad \mbox{for $t \geq 0$}.
\ee
Using the bound \eqref{ineq:NNR_mass}, we deduce from \eqref{ineq:VVmass_critical} that
$$
\VV_R[u(t)] \leq 16 E[u_0] t^2 + C_1  (1+t)^{b \nu + 1} + C_2 (1+t) \quad \mbox{for $t \geq 0$},
$$
with some constants $C_1=C_1(u_0, R, b, \nu)>0$ and $C_2=C_2(u_0)> 0$. Suppose now that
$$
b \nu < 1.
$$
Since $E[u_0] < 0$ by assumption, we see that $\VV_R[u(t_*)] < 0$ for some sufficiently large time $t_* > 0$. But this is a contradiction. Hence the bound \eqref{ineq:infinite} cannot hold if $b \nu < 1$. Therefore, we conclude
\be \label{eq:limsup}
\limsup_{t \to +\infty} \left ( t^{-\nu} \| \DD u(t) \|_{L^2} \right ) = +\infty,
\ee
provided that
$$
0 \leq \nu < \begin{dcases*} +\infty & for $d \geq 8$, \\ 24 & for $d =7$, \end{dcases*} \qquad 0 \leq \nu  < \begin{dcases*} 10 & for $d =6$, \\ \frac{4}{6-d} & for $d=3,4,5$. \end{dcases*}
$$
For $d \geq 5$, we note that \eqref{eq:limsup} gives extra information that cannot be deduced from the lower bound \eqref{ineq:infiniteblowup}.

The proof of Theorem \ref{thm:mass} is now complete. \hfill $\blacksquare$

\begin{appendix}

\section{Ground States for Biharmonic NLS} \label{sec:Q}

\subsection{Energy-Subcritical Case.}
Let $d \geq 1$ and assume that $0 < \sigma < \sigma_*$, where $\sigma_* = +\infty$ if $d \leq 4$ and $\sigma_* = \frac{4}{4-d}$ if $d \geq 5$. For $u  \in H^2(\R^d)$ with $u \not \equiv 0$, we define the Weinstein functional
\be
\WW_{d,\sigma}[u] := \frac{ \| u \|_{L^{2 \sigma +2}}^{2 \sigma+2}}{ \| \Delta u \|_{L^2}^{ \frac{d \sigma}{2}}  \| u \|_{L^2}^{ 2 \sigma+2- \frac{d \sigma}{2} }} ,
\ee
and we consider the corresponding maximization problem given by
\be \label{def:Wein}
C_{d,\sigma} := \sup_{0 \not \equiv u \in H^2(\R^d)} \WW_{d,\sigma}[u] .
\ee 
It can be shown this supremum is attained; see, e.\,g., \cite{BeFrVi-14} and also below for a simple proof when $\sigma \in \N$. By construction, the number $C_{d,\sigma}>0$ is the optimal constant for the Gagliardo-Nirenberg (GN) interpolation inequality
\be \label{ineq:GNSAp}
\| u \|_{L^{2 \sigma+2}}^{2 \sigma+2} \leq C_{d,\sigma}   \| \Delta u \|_{L^2}^{ \frac{d \sigma}{2}}  \| u \|_{L^2}^{ 2 \sigma + 2 - \frac{d\sigma}{2} } 
\ee
valid for all $u \in H^2(\R^d)$. Following standard convention, we say that $0 \not \equiv Q \in H^2(\R^d)$ is a {\bf ground state} if $Q$ optimizes \eqref{ineq:GNSAp}; or, equivalently, if $Q$ is a maximizer for \eqref{def:Wein}.  A calculation shows that any ground state $Q \in H^2(\R^d)$ must satisfy (after a rescaling $Q \mapsto \mu Q(\lambda \cdot)$ with suitable constants $\mu, \lambda > 0$) the nonlinear equation
\be \label{eq:Q}
\DD^2 Q + Q - |Q|^{2 \sigma} Q = 0 \quad \mbox{in $\R^d$}.
\ee
It should be remarked that (in contrast to NLS with $\DD$ instead of $\DD^2$) radial symmetry of ground states $Q$ is not known.  However, what is known is that, if $Q$ is assumed to be radial and real-valued, then positivity of $Q$ cannot hold, since an asymptotic expansion shows that $Q(r)$ changes its sign infinitely often as $r \to \infty$; see, e.\,g., \cite{FiIlPa-02}. In general, the delicate issue of uniqueness  of $Q$ (modulo symmetries) as well as the non-degeneracy of the associated linearized operator are completely open questions.

\begin{prop}[Pohozaev-Type Identities] Let $d \geq 1$ and $0 < \sigma < \sigma_*$. For any solution $Q \in H^2(\R^d)$ of  \eqref{eq:Q}, we have
$$
\| \DD Q \|_{L^2}^2 =  \left( \frac{d}{d+2(2-s_c)} \right ) \| Q \|_{L^{2 \sigma+2}}^{2 \sigma+2} = \left ( \frac{d}{2 (2- s_c)} \right ) \| Q \|_{L^2}^2
$$
with $s_c = \frac{d}{2} - \frac{2}{\sigma}$. If moreover $Q \in H^2(\R^d)$ is a ground state, then
$$
K_{d,\sigma} = \| \DD Q \|_{L^2}^{s_c} \| Q \|_{L^2}^{2-s_c} = \left ( \frac{s_c}{d} \right)^{-\frac{s_c}{d}} E[Q]^{\frac{s_c}{2}} M[Q]^{1-\frac{s_c}{2}},
$$
where 
$$
K_{d,\sigma} = \left ( \frac{4 (\sigma+1)}{d \sigma C_{d,\sigma}} \right )^{\frac{1}{\sigma}}.
$$
\end{prop}

\begin{proof}
If we integrate equation \eqref{eq:Q} against $\overline{Q}$ and $x \cdot \nabla \overline{Q}$, we find
$$
\| \DD Q \|_{L^2}^2 + \| Q \|_{L^2}^2 - \| Q \|_{L^{2 \sigma+2}}^{2 \sigma+2} = 0,
$$
$$
(4-d) \| \DD Q \|_{L^2}^2 - d \| Q \|_{L^2}^2 +\frac{2d}{2 \sigma+2} \| Q \|_{L^{2\sigma+2}}^{2 \sigma+2} = 0.
$$
Note that, by standard arguments, we check that $x \cdot \nabla Q$ has sufficient regularity and spatial decay that justifies this calculation. The rest of the proof follows from direct computations, using also that a ground state $Q \in H^2(\R^d)$, which, by definition, turns \eqref{ineq:GNSAp} into an equality.
\end{proof}

\subsection{Radial Symmetry of Ground States}
The aim of this subsection is to prove a radial symmetry result for ground states $Q$ for the biharmonic NLS. To the best of our knowledge, nothing is known in that respect. We present an argument based on symmetric-decreasing rearrangement in Fourier space. By using this approach, we will be able to treat the case of integer exponents $\sigma \in \N$.

For $u \in L^2(\R^d)$, we define its {\bf Fourier rearrangement} to be given by
$$
u^{\sharp} := \mathcal{F}^{-1} \{ ( \mathcal{F} u)^* \},
$$
where $f^*$ denotes the symmetric-decreasing rearrangement of a measurable function $f : \R^d \to \C$ that vanishes at infinity, i.\,e., the level sets $\{  |f(x)| > t \} \subset \R^d$ have finite (Lebesgue) measure for every $t>0$; see, e.\,g., \cite{LiLo-01} for a review of rearrangement techniques. For the Fourier transform $\mathcal{F}$, we use the convention that
$$
(\mathcal{F} u)(\xi) = \int_{\R^d} u(x) e^{-2\pi i x \cdot \xi} \, dx,
$$
and thus the inverse Fourier transform is given by $(\mathcal{F}^{-1} v)(x) = \int_{\R^d} v(\xi) e^{2 \pi i x \cdot \xi} \, d\xi$. Note that we always have that $\| u^\sharp \|_{L^2} = \| u \|_{L^2}$ by Plancherel's theorem and the fact that $\| f^* \|_{L^2} = \| f \|_{L^2}$. Furthermore, the function $u^\sharp(x)$ is radially symmetric, since it is the (inverse) Fourier transform of the radially symmetric function $(\FF u)^*$ on $\R^d$. 


\begin{lemma} \label{lem:rearrange} For $d \geq 1$, the following inequalities hold.
\begin{enumerate}
\item[(i)] If $u \in H^s(\R^d)$ with $s \geq 0$, then $u^\sharp \in H^s(\R^d)$ and
$$
\| (-\DD)^s u^\sharp \|_{L^2} \leq \| (-\DD)^s u \|_{L^2} .
$$
Moreover, for $s>0$, we have equality if and only if $|\widehat{u}| = |\widehat{u}|^*$.
\item[(ii)] Let $m \geq 1$ be an integer. If $u \in L^1(\R^d) \cap L^{2m}(\R^d)$ with $\FF u \in L^1(\R^d)$, then $u^{\sharp} \in L^{2m}(\R^d)$ and
$$
\| u \|_{L^{2m}} \leq \| u^\sharp \|_{L^{2m}} .
$$ 
\end{enumerate}
\end{lemma}

\begin{proof} To prove assertion (i), we first note that $u^\sharp \in L^2(\R^d)$, since $\FF u \in L^2(\R^d)$ and $(\FF u)^* \in L^2(\R^d)$. Thus we have $\mathcal{F}(u^\sharp) = (\FF{u})^*$. Next, we recall the well-known property of the symmetric-decreasing rearrangement that $(|f|^2)^* = (|f|^*)^2$. Thus, by Plancherel's theorem, the claimed inequality in (i) is equivalent to the estimate
\be \label{ineq:rearrange1}
\int_{\R^d}  f^*(y) | 2 \pi y|^{2s} \, dy \leq \int_{\R^d} f(y) | 2 \pi y|^{2s} \, dy
\ee
for any nonnegative measurable function $f \geq 0$ on $\R^d$ that vanishes at infinity. By the layer cake representation, we can write $f(y) = \int_0^\infty \chi_{\{ f > t\}}(y) \, dt$ for almost every $y \in \R^d$ (see, e.\,g., \cite{LiLo-01}). Therefore, it suffices to prove that
\be \label{ineq:rearrange2}
\int_{\R^d} \chi_{A^*}(y) |2 \pi y|^{2s}  \, dy \leq \int_{\R^d} \chi_A(y) | 2 \pi y|^{2s} \, dy
\ee
for any measurable set $A \subset \R^d$ with finite measure, where $A^*$ denotes the symmetric-decreasing rearrangement of $A$, i.\,e., the set $A^*=B_R(0) \subset \R^d$ is the (open) ball around the origin with radius $R>0$ such that $\mu(B_R(0)) = \mu(A)$. (If $\mu(A) = 0$ we take $A^*= \emptyset$.) The proof of \eqref{ineq:rearrange2} is a simple exercise in measure theory. For the reader's convenience, we give the details here. From $\mu(A \setminus A^*) = \mu(A) - \mu(A \cap A^*)$, $\mu(A^* \setminus A) = \mu(A^*)- \mu (A \cap A^*)$, and $\mu(A) = \mu(A^*)$, we deduce that 
$\mu(A \setminus A^*) = \mu(A^* \setminus A)$. This gives us 
\be \label{ineq:rearrange3}
\int_{A\setminus A^*} |y|^{2s} \, dy \geq R^{2s} \mu(A\setminus A^*) = R^{2s} \mu(A^* \setminus A) \geq \int_{A^* \setminus A} |y|^{2s} \, dy,
\ee
using that $|y|^{2s}$ is monotone increasing in $|y|$. Hence $\int_{A} |y|^{2s}  = \int_{A \setminus A^*} |y|^{2s}  + \int_{A \cap A^*} |y|^{2s} \geq \int_{A^* \setminus A} |y|^{2s} + \int_{A \cap A^*} |y|^{2s}  = \int_{A^*} |y|^{2s}$. This shows \eqref{ineq:rearrange2} and hence \eqref{ineq:rearrange1}, which yields in particular that $u^\sharp \in H^s(\R^d)$.

Now suppose that $s>0$  and that equality in \eqref{ineq:rearrange3} holds. In particular, we have the strict inequality $|y|^{2s} < R^{2s}$ for $y \in A^*=B_R$. Suppose now that $\mu(A^* \setminus A) > 0$. Then  $\int_{A^* \setminus A} |y|^{2s} \, dy < R^{2s} \mu(A^*\setminus A)$, but this gives a contradiction if equality holds in \eqref{ineq:rearrange3}. Thus we conclude that equality in \eqref{ineq:rearrange3} can hold only if $\mu(A^* \setminus A)= \mu(A \setminus A^*)=0$, which means $\mu(A \cap A^*) =0$, since $\mu(A) = \mu(A^*)$. In summary, we deduce that equality in \eqref{ineq:rearrange1} can only hold if the level sets of $f \geq 0$ satisfy $\{ f > t \} = \{ f > t\}^*$ (up to a zero measure set) for almost every $t>0$. If we apply this to $f=|\FF{u}|$, we complete the proof of (i).

We now turn to the proof of (ii). We start by showing that $u^\sharp \in L^{2m}(\R^d)$ as follows. Since $u \in L^1(\R^d) \cap L^{2m}(\R^d)$ with $m \geq 1$, we have $u \in L^2(\R^d)$. Consequently, $\FF u \in L^2(\R^d)$ and therefore $(\FF u)^* \in L^2(\R^d)$. Also, since $\FF u \in L^1(\R^d)$ by assumption, it holds that $(\FF u)^* \in L^1(\R^d)$. Thus $(\FF u)^* \in L^1(\R^d) \cap L^2(\R^d)$, which implies that $u^\sharp \in L^2(\R^d) \cap L^\infty(\R^d)$, which shows that $u^\sharp \in L^{2m}(\R^d)$.

Next, because $2m$ is an even integer, we can write
$$
\int_{\R^d} |u(x)|^{2m} \, dx =  \FF{(|u|^{2m})}(0)  = ( \FF u \star \FF \overline{u} \star \ldots \star \FF u \star \FF \overline{u})(0),
$$
using the convolution theorem $\FF(fg) = \FF f \star \FF g$ for $\FF f, \FF g \in L^1(\R^d) \cap L^2(\R^d)$ iteratively $m-1$ times. Now, by the Brascamp-Lieb-Luttinger inequality \cite{BrLiLu-74} (the generalized Riesz' rearrangement inequality), we have that
\begin{align*}
 ( \FF u \star \FF \overline{u} \star \ldots \star \FF u \star \FF \overline{u})(0) & \leq  ( (\FF u)^* \star (\FF \overline{u})^* \star \ldots \ast (\FF u)^* \star (\FF \overline{u})^*)(0) \\
 & = ( (\FF u)^* \star (\FF u)^* \star \ldots \star (\FF u)^* \star (\FF u)^*)(0).
\end{align*}
In the last step, we used the fact that $( \FF \overline{u})^* = (\FF u)^*$, since the functions $(\FF u)(\xi)$ and $(\FF \overline{u})(\xi) = \overline{\FF{u}(-\xi)}$ are equimeasurable. Next, we recall that $(\FF u)^* = \FF (u^\sharp)$ and $\widehat{(u^\sharp)} \in L^1(\R^d) \cap L^2(\R^d)$. Applying the convolution theorem again, we deduce that  $$
\| u \|_{L^{2m}}^{2m} \leq ( \FF{(u^\sharp)} \star \FF{(u^\sharp)} \star \ldots \star \FF {(u^\sharp)} \star \FF{(u^\sharp)})(0) = \FF{(|u^\sharp|^{2m})}(0)= \| u^\sharp\|_{L^{2m}}^{2m},$$
 whence assertion (ii) follows. The proof of Lemma \ref{lem:rearrange} is now complete. 
\end{proof}

\begin{prop}[Radial Symmetry  of Ground States.] Let $d \geq 1$, $0 < \sigma < \sigma_*$, and assume also that $\sigma \in \N$. Then there exists a ground state $Q \in H^2(\R^d)$ with $Q = Q^\sharp$. As a consequence of this, the following properties hold.

\begin{enumerate}
\item[(i)] $Q(x)$ is radially symmetric, real-valued, and continuous.
\item[(ii)] $Q(0) \geq |Q(x)|$ for all $x \in \R^d$.
\end{enumerate}
\end{prop}

\begin{proof}
Let $Q \in H^2(\R^d)$ be a ground state, i.\,e., a maximizer for problem \eqref{def:Wein}. We claim that its Fourier transform $\widehat{Q}:=\FF Q$ belongs to $L^1(\R^d)$. Without loss of generality we can assume that $Q$ solves \eqref{eq:Q}. By iterating the associated integral equation 
$Q = (\DD^2+1)^{-1} ( Q \overline{Q})^\sigma Q$ using Sobolev's inequalities and that $\sigma$ is an integer, we find that $Q \in H^k(\R^d)$ for all $k \in \N$. In particular, we can choose an integer $k > d/2$ to conclude $\| \widehat{Q} \|_{L^1} \leq \| \langle \xi \rangle^{-k} \|_{L^2} \| \langle \xi \rangle^{k} \widehat{Q} \|_{L^2} \leq C \| Q \|_{H^k} < \infty$.

Thus, we can apply Lemma \ref{lem:rearrange} to $Q$ with $m=\sigma+1 \in \N$ to deduce that $\WW_{d,\sigma}[Q^\sharp] \geq \WW_{d, \sigma}[Q]$ and hence $Q^\sharp \in H^2(\R^d)$ maximizes \eqref{def:Wein} too. Therefore, we can choose $Q=Q^\sharp$ to be a ground state for \eqref{def:Wein}. 

The rest of the proof follows from Bochner's theorem (see, e.\,g.,~\cite{ReSi-75}). Since $Q = Q^\sharp$ has a nonnegative Fourier transform $\FF(Q^\sharp)(\xi) \geq 0$ with $\FF(Q^\sharp) \in L^1(\R^d)$, we deduce that $Q : \R^d \to \C$ is a positive-definite function. That is, $Q$ is a bounded and continuous function with the following property: For every integer $m \geq 1$ and any points $x_1, \ldots, x_m \in \R^d$, the matrix $(Q(x_i-x_j))_{i,j =1}^{m}$ is positive semi-definite on $\C^m$, i.\,e.,
$$
\sum_{i,j=1}^m Q(x_i-x_j) \overline{\zeta_i} \zeta_j \geq 0 \quad \mbox{for all $\zeta \in \C^m$.}
$$   
If we take $m=1$ and $x_1=0$, we deduce that $Q(0)$ is real with $Q(0) \geq 0$. Moreover, by taking $m=2$ with $x_1 = 0$ and $x_2 = x$ with arbitrary $x \in \R^d$ (and considering the vectors $\zeta = ( Q(x), i Q(x)) \in \C^2$ and $\zeta = ( i Q(0), Q(x)) \in \C^2$), we conclude that
\be \label{eq:bochner}
Q(0)^2 \geq |Q(x)|^2 \quad \mbox{and} \quad Q(-x) = \overline{Q(x)}.
\ee
Next, since the Fourier transform $\mathcal{F}(Q)(\xi)$ is radially symmetric in $\xi$, we deduce that $Q(x)$ is radially symmetric in $x\in \R^d$. In view of the second equation in \eqref{eq:bochner}, this implies that $Q(x)$ must be real-valued.  This completes  the proof of Proposition A.2.
\end{proof}

\begin{remark*}
{\em The previous symmetry result also provides a simple existence proof for ground states for integer $\sigma \in \N$ and $d \geq 2$ as follows. Indeed, let $(u_n)_{n \geq 1} \subset H^2(\R^d)$ be a maximizing sequence for problem \eqref{def:Wein}, normalized such that $\| u_n \|_{L^2} = \| \DD u_n \|_{L^2}=1$ for all $n$. By density, we can assume that $u_n \in \mathcal{S}(\R^d)$ are Schwartz functions for all $n \geq 1$. From Lemma \ref{lem:rearrange} we have $\WW_{d,\sigma}[u_n^\sharp] \geq \WW_{d,\sigma}[u_n]$ and hence we can replace $u_n$ by $u_n^\sharp$. Without loss of generality, we can renormalize such that $\| u^\sharp_n \|_{L^2}= \| \DD u^{\sharp}_n \|_{L^2}=1$. Since $u^\sharp_n$ are radial functions uniformly bounded in $H^1(\R^d)$, an application of the Strauss inequality \eqref{ineq:Strauss} now yields a uniform spatial decay for this sequence and we easily deduce that the sequence $u^\sharp_n$ converges (up to subsequences) strongly in $H^2(\R^d)$ to a maximizer for problem \eqref{def:Wein}. }
\end{remark*}

\subsection{Energy-Critical Case} Let $d \geq 5$. We recall the Sobolev inequality
\be \label{ineq:sob2}
\| u \|_{L^{\frac{2d}{d-4}}} \leq C_d \| \DD u \|_{L^2}
\ee 
for all $u \in \dot{H}^2(\R^d)$, where $C_d > 0$ denotes the optimal constant. We recall the following result about existence and uniqueness of optimizers.

\begin{lemma}
For $d \geq 5$, we have equality in \eqref{ineq:sob2} if and only if $u(x) = \lambda W(\mu(x-x_0))$ for some $\lambda \in \C$, $\mu > 0$, and $x_0 \in \R^d$, where 
$$
W(x) = \left (  \frac{(d(d-4)(d^2-4))^{\frac{1}{4}}}{1 +x^2} \right )^{\frac{d-4}{2}} .
$$
\end{lemma}

\begin{proof}
If we let $f = (-\DD)^{-1} u$, we see that \eqref{ineq:sob2} is equivalent to the following instance of the weak Young (or Hardy-Littlewood-Sobolev) inequality $\| (-\DD)^{-1} f \|_{L^\frac{2d}{d-4}} \leq C_d \| f \|_{L^2}$.  Uniqueness of optimizers and the explicit form of $W(x)$ now follows from Lieb's result \cite{Lieb-83}; see also \cite{CaLo-90} for a different approach using the method of competing symmetries.
\end{proof}

A calculation shows that the optimizer $W(x)$ from above solves the equation
\be \label{eq:W2}
\DD^2 W - |W|^{\frac{8}{d-4}} W = 0 \quad \mbox{in $\R^d$}.
\ee 
Let us also mention the symmetry results in \cite{ChLiOu-06,Li-06}, where it is shown that any nonnegative solution of \eqref{eq:W2} in $L^{2d/(d-4)}_{\mathrm{loc}}(\R^d)$ equals $W(x)$ up to translation and rescaling.
 
Finally, we derive some Pohozaev identities for $W$ as follows. Integrating equation \eqref{eq:W2} against $W$ yields $\| \DD W \|_{L^2}^2 = \| W \|_{L^p}^p$ with $p=\frac{2d}{d-4}$. Since $W$ optimizes \eqref{ineq:sob2}, we also have $\| W \|_{L^p}^2 = C_d^2 \| \DD W \|_{L^2}^2$. Thus, we find the Pohozaev identities
\be \label{eq:Wpoho}
\| \DD W \|_{L^2}^2 = \left ( \frac{1}{C_d} \right )^{\frac{d}{2}} \quad \mbox{and} \quad E[W] = \left ( \frac{1}{2} - \frac{d-4}{2d} \right ) \| \DD W \|_{L^2}^2 = \frac{2}{d} \| \DD W \|_{L^2}^2.
\ee

\section{On the Choice of Cutoff Functions} \label{sec:cutoff}

Let $\phi : \R^d \to \R$ be a cutoff function as in Section \ref{sec:locvirial}. It is easy to see that we can choose $\phi(r) \geq 0$ to be nonnegative for all $r \geq 0$ with compact support such that $\phi(r) \equiv 0$ for $r \geq 10$. Furthermore, we can choose $\phi(r) \geq 0$ such that $\nabla^j \sqrt{\phi} \in L^\infty(\R^d)$ for $0 \leq j \leq 6$. Hence the additional properties \eqref{eq:phi_Riesz} for $\phi(r)$ used in Section \ref{sec:locrieszvar}  hold.

Let us now discuss that we can choose $\phi(r)$ with some further additional properties used in the proof of Theorem \ref{thm:mass} for $\mu=0$. In particular, we need to choose $\phi(r)$ such that inequality \eqref{ineq:trans} holds for $\eta_0 > 0$ sufficiently small, i.\,e., we have
\be \label{ineq:trans_A}
1- \pt_r^2 \phi_R(r) - \eta_0 \left \{ R^4 (A_R(r))^2 + (B_R(r))^{\frac{d}{2}} \right \} \geq  0 \quad \mbox{for $r \geq 0$}.
\ee
Recall that $A_R(r) = 4 \pt_r^2 \Delta \phi_R + 2 \DD^2 \phi_R(r)$ and $B_R(r) = \frac{8d}{4+d} \left ( d- \DD \phi_R(r) \right ) \geq 0$. Since $\phi_R= R^2 \phi(r/R)$, the claimed lower bound \eqref{ineq:trans_A} is equivalent to
\be \label{ineq:trans_A2}
1- \phi''(r) - \eta_0 \left \{ 4 \left (\DD^2 \phi(r) + 4 \pt_r^2 \Delta \phi(r) \right )^2 + \left ( \frac{8d}{4+d} \right )^{\frac{d}{2}} (d-\DD \phi(r))^{\frac{d}{2}} \right \} \geq 0
\ee
for $r \geq 0$. Let us now take $\phi(r) \geq 0$ such that
$$
\phi'(r) = \begin{dcases*} r & for $0 \leq r \leq 1$, \\ r - (r-1)^6 & for $1 < r \leq 1 + 1/\sqrt[5]{6}$, \\
\mbox{$\phi'(r)$ smooth with $\phi''(r) \leq 0$} & for $1+1/\sqrt[5]{6} < r \leq 10$, \\
0 & for $r \geq 10$. \end{dcases*}
$$
Because $1- \phi''(r) \equiv \DD^2 \phi(r) \equiv \pt_r^2 \Delta \phi(r) \equiv d-\DD \phi_R(r) \equiv 0$ for $0 \leq r \leq 1$, it remains to show that \eqref{ineq:trans_A2} holds for $r > 1$. Since we have
$$
1-\phi''(r) \geq 1 \quad \mbox{and} \quad | \DD^2 \phi(r)| + |\pt_r^2 \Delta \phi(r)| + |d-\DD \phi(r)| \leq C \quad \mbox{for $r \geq 1+1/\sqrt[5]{6}$},
$$
we can find $\eta_0> 0$ sufficiently small such that \eqref{ineq:trans_A2} is true for $r \geq 1 + 1/\sqrt[5]{6}$. In the region $1 < r < 1 + 1/\sqrt[5]{6}$, a computation yields $1- \phi''(r) = 6(r-1)^5$ and
$$
(\pt_r^2 \Delta \phi(r))^2 + (\DD^2 \phi(r))^2  \leq C (r-1)^6 , \quad |d-\DD \phi(r)|^{\frac{d}{2}} \leq C (r-1)^{5 \cdot \frac{d}{2}}.
$$ 
Since $d \geq 2$, we deduce that we can choose $\eta_0 > 0$ sufficiently small to ensure that \eqref{ineq:trans_A2} holds for $1 < r < 1 + 1/\sqrt[5]{6}$ as well. 

Finally, with the choice of $\phi(r)$ above, we have that  $B_R(r) = B(r/R)$, where
$$
B(r) = \begin{dcases*} 0 & for $0 \leq r \leq 1$, \\ 
\frac{8d}{4+d} (r-1)^5 \left ( 6 + \frac{(d-1)(r-1)}{r} \right ) & for $1 < r \leq 1+ 1/\sqrt[5]{6}$, \\
\end{dcases*} 
$$
and $B(r)$ is smooth for $r > 1+ 1/\sqrt[5]{6}$ with $B(r) \equiv \mbox{const}.$ for $r \geq 10$. Since $d \geq 2$, we deduce the bounds
$\| \partial_r^j B_R^{\frac{d}{4}} \|_{L^\infty} \lesssim R^{-j}$ for $j=1,2$ and $\| \partial_r B_R^{\frac{d}{8}} \|_{L^\infty} \lesssim R^{-1}$.
\end{appendix}

\bibliographystyle{amsplain}
\bibliography{BNLSBib}

\end{document}